\documentclass[a4paper]{article}
\usepackage[english]{babel}
\usepackage[utf8x]{inputenc}
\usepackage[T1]{fontenc}
\usepackage{caption}

\usepackage[a4paper,top=3cm,bottom=2cm,left=3cm,right=3cm,marginparwidth=1.75cm]{geometry}
\usepackage{subcaption}
\usepackage{amsmath}
\usepackage{graphicx}
\usepackage{multirow}

\usepackage{hhline}
\usepackage{mathrsfs} 
\usepackage{amsmath}
\usepackage{amssymb}
\usepackage{amsthm}
\usepackage{dsfont}

\usepackage{multicol}
\usepackage{nomencl}
\makenomenclature

\renewcommand\nompreamble{\begin{multicols}{2}}
\renewcommand\nompostamble{\end{multicols}}

\usepackage{float}
\usepackage{graphicx,subcaption}
\usepackage{color}
\usepackage{tabu}

\usepackage{amsmath}

\makeatletter 
\renewcommand{\fnum@figure}{\small\textbf{\figurename\thefigure}} 
\makeatother
\usepackage[algoruled,vlined,boxed,longend,english,lined,boxed,commentsnumbered]{algorithm2e}
\numberwithin{equation}{section}%
\newtheorem{theorem}{Theorem}[section]
\newtheorem{lemma}{Lemma}[section]
\newtheorem{proposition}{Proposition}[section]

\newtheorem{remark}{Remark}[section]

\newcommand{\dx} {\displaystyle\mathrm{d}x}

\newcommand{\ds} {\displaystyle\mathrm{d}s}

\def\ds{\displaystyle}

\usepackage{lineno}
\usepackage{diagbox}

\title{\textbf{Identification of the initial value for a space-time fractional diffusion equation}}
\author{Mohamed BenSalah\thanks{ Mohamed BenSalah,\hfil\break
Department of Mathematics, ISSAT of Sousse,\newline
University of Sousse, Rue Tahar Ben Achour, Sousse 4003, Tunisia\hfil\break
E-mail : mohamed.bensalah@fsm.rnu.tn} \, and  Salih Tatar\thanks{ Salih Tatar,\hfil\break
Department of Mathematics $\&$  Computer Science, College of Science and General Studies\newline
Alfaisal University, Riyadh, KSA,\hfil\break
E-Mail : statar@alfaisal.edu}}

\begin{document}
\maketitle
\begin{abstract}
In this paper, we study an inverse problem for identifying the initial value in a space-time fractional diffusion equation from the final time data. We
show the identifiability of this inverse problem by proving the existence of its
unique solution with respect to the final observed data. It is proved that the inverse problem is an ill-posed problem. Namely, we prove that the solution to the inverse problem does not depend continuously on the measured data. The inverse problem is
formulated as a regularized optimization one minimizing a least-squares type
cost functional. Then the conjugate gradient method combined with Morozov's discrepancy is proposed for finding a stable approximate solution to the regularized variational problem. Numerical examples with noise-free and noisy data illustrate the applicability and high accuracy of the proposed method to some extent. \\


\end{abstract}
\nomenclature{$T : $}{Final time} \nomenclature{$\alpha : $}{Fractional order of the time derivative}
\nomenclature{$\Omega: $}{Domain of $\mathbb{R}^d$, $d\geq 1$}  \nomenclature{$s :$}{Fractional order of the space derivative}
\nomenclature{$\Gamma :$}{Gamma function}  \nomenclature{$\delta :$}{Dirac delta function}
\nomenclature{$L^2(\Omega ) :$}{Set of square integrable functions on $\Omega$} \nomenclature{$H^\alpha(0,T) :$}{Fractional Sobolev space in time}
\nomenclature{$\alpha_k :$}{Conjugate coefficient in the $k$th iteration} \nomenclature{$d_k :$}{Descent direction in the $k$th iteration}
\nomenclature{$\zeta_k :$}{Step size in the $k$th iteration}
\nomenclature{$\theta :$}{Noise level}
\nomenclature{$\Omega^c :$}{Complement of $\Omega$ in $\mathbb{R}^d$}
\nomenclature{$\mu :$}{Relative noise level}
\nomenclature{$E_k :$}{$L^2$ error in the $k$th iteration}
\nomenclature{$\lvert \, . \, \rvert_{H^s(\Omega) } :$}{Aronszajn-Slobodeckij seminorm}
\nomenclature{$\lVert \, . \, \rVert_{L^2(\Omega) } :$}{$L^2$ norm in $\Omega$}
\nomenclature{$\lVert \, . \, \rVert_{H^s(\Omega) } :$}{$H^s$ norm in $\Omega$}
\nomenclature{$H^s(\Omega) :$}{Fractional Sobolev space on $\Omega$}
\nomenclature{$H^s(\mathbb{R}^d) :$}{Fractional Sobolev space on $\mathbb{R}^d$}
\nomenclature{$\mathcal{R}_k :$}{Residual in the $k$th iteration}
\nomenclature{$\mathcal{T}_\nu :$}{Objective function}
\nomenclature{$\mathcal{I}_s :$}{Stopping index}
\nomenclature{$E_{\alpha, \beta} :$}{Mittag-Leffler function}
\nomenclature{$\widetilde{H}^s(\Omega) :$}{Set of functions in $H^{s}\left(\mathbb{R}^{d}\right)$ supported within $\overline{\Omega}$}
\nomenclature{$\overline{\Omega} :$}{Closure of $\Omega$}
\nomenclature{$C_{d,s} :$}{Normalization constant}
\nomenclature{\textbf{P.V} :}{Principal value}
\nomenclature{$\Delta t :$}{Time step}
\nomenclature{$\Delta x :$}{Space step}
\nomenclature{{\it rand}$(.)$ :}{Random function}
\nomenclature{$h :$}{Measured data}
\nomenclature{$\mathcal{B}_s$: }{Bilinear form associated to the space $\widetilde{H}^s(\Omega)$}
\nomenclature{$\chi_{(-1,1)}$: }{Characteristic function of the interval $(-1,1)$}

\printnomenclature

\section{Introduction}
Let $T>0$, $\Omega \subset \mathbb{R}^d$ and $d\geq 1$ be an open bounded domain. We consider the following initial-boundary value problem for a space-time fractional
diffusion equation with the homogeneous Dirichlet boundary and initial conditions:
\begin{equation}\label{space_time}
\left\{\begin{aligned}
\partial_t^\alpha u = -(-\Delta)^s u, &  & & (x, t) \in \Omega \times(0, T), \\
u & = 0, & &  (x, t) \in  \Omega^c \times(0, T),\\
u & = g, & &  (x, t) \in \Omega \times\{0\},
\end{aligned}\right.
\end{equation}
where $\Omega^c$ is the complement of $\Omega$ in $\mathbb{R}^d$, $g(x) \in L^2(\Omega)$ is an initial function ,  $\partial_{t}^{\alpha}$ denotes the left Caputo fractional derivative of order $0<\alpha<1$ (see, e.g., \cite{kilbas2006theory}), defined by
\begin{equation}\label{caputto}
\partial_{t}^{\alpha} v(t):=\ds \frac{1}{\Gamma(1-\alpha)} \int_{0}^{t} \frac{v^{\prime}(\tau)}{(t-\tau)^{\alpha}} \mathrm{d} \tau,
\end{equation}
 and $(-\Delta)^{s}$ is the non-local fractional Laplacian operator of order $s\in (0,\,1)$, defined by
\begin{equation}\label{laplace}
(-\Delta)^{s} u(x)=C_{d, s} \text { P.V. } \int_{\mathbb{R}^{d}} \frac{u(x)-u(y)}{|x-y|^{d+2 s}} d y.
\end{equation}
In \eqref{caputto} and \eqref{laplace}, $\Gamma(.)$ is the Gamma function,  $C_{d, s}$ is a normalization constant, defined by
$$
C_{d, s}=\frac{2^{2 s} s \Gamma\left(s+\frac{d}{2}\right)}{\pi^{d / 2} \Gamma(1-s)}, 
$$
and "P.V." is the principal value of the integral, defined by
$$
\text { P.V. } \int_{\mathbb{R}^{d}} \frac{v(x)-v(y)}{|x-y|^{d+2 s}} d y=\lim _{\varepsilon \downarrow 0} \int_{\left\{y \in \mathbb{R}^{d},|y-x|>\varepsilon\right\}} \frac{v(x)-v(y)}{|x-y|^{d+2 s}} d y. 
$$

\begin{remark}
The fractional Laplace operator is also defined by either Spectral/Fourier definition, or standard Laplacian, or  L$\acute{e}$vy process, or directional representation.   In this work, we use the fractional Laplace operator by singular integral definition as in \eqref{laplace}. We refer the readers to \cite{di2012hitchhiker} and \cite{lischke2020fractional} for more details regarding the Fractional Laplacian operator.
\end{remark}

\noindent For given $(\alpha, s) \in (0,1)^2$, and $g(x) \in L^2(\Omega)$, the problem \eqref{space_time} is called the direct(forward) problem. Like most direct problems of mathematical physics, the problem \eqref{space_time}  is well-posed, see for example, \cite{salih1} and \cite{nane}. The inverse problem here consists of determining the function $g(x)$, by means of the observation data (additional data) $u(x,T)=h(x), \, x \in \Omega.$ \\

The space-time fractional diffusion equation $\partial_t^\alpha u = -(-\Delta)^s u $ with $0 < \alpha, s  < 1$ is used to model anomalous diffusion \cite{MBSB}. Here, the fractional derivative in time is used to describe particle sticking and trapping phenomena, and the fractional space derivative is used to model long particle jumps. These two effects combined together produce a concentration profile with a sharper peak and heavier tails. In the fractional diffusion equations, the fractional time derivative with $0 < \alpha < 1$ is  used to model slow diffusion, and the exponent $s$  is related to the parameter specifying the large-time behavior of the waiting-time distribution function and particular cases of the infinitesimal generators of L$\acute{e}$vy stable diffusion processes and appear in anomalous diffusions in plasmas, flames propagation and chemical reactions in liquids, population dynamics, geophysical fluid dynamics, see \cite{11}, \cite{KU}, \cite{di2012hitchhiker} and some of the references cited therein. Recently, there has been a growing interest in inverse problems with fractional derivatives. These problems are physically and practically very important. For example in  \cite{y1}, the authors prove a uniqueness result in a one-dimensional time-fractional diffusion equation. The proof is based on the eigenfunction expansion of the weak solution to the initial value/boundary value problem and the Gel'fand–Levitan theory. In  \cite{rundell}, they study an inverse problem of recovering a spatially varying potential term in a one-dimensional time-fractional diffusion equation from the flux measurements taken at a single fixed time corresponding to a given set of input sources. The unique identifiability of the potential is shown for two cases, i.e. the flux at one end and the net flux, provided that the set of input sources forms a complete basis in $ L^2(0, 1)$. An algorithm of the Quasi-Newton type is proposed for the efficient and accurate reconstruction of the coefficient from finite data, and the injectivity of the Jacobian is discussed. In \cite{y2}, the authors consider a backward problem in time for a time-fractional partial differential equation in the one-dimensional case, which describes the diffusion process in porous media related to the continuous time random walk problem. The backward problem is ill-posed and they propose a regularizing scheme by the quasi-reversibility with fully theoretical analysis and test its numerical performance. With the help of the memory effect of the fractional derivative, it is found that the property of the initial status of the medium can be recovered in an efficient way. In \cite{y3}, for a time-fractional diffusion equation with source term, they discuss an inverse problem of determining a spatially varying function of the source by final overdetermining data. They prove that this inverse problem is well-posed in the Hadamard sense except for a discrete set of values of diffusion constants.  In \cite{y4}, the authors consider initial value/boundary value problems for a fractional diffusion-wave equation. For $\alpha \in (0,1)$, they prove stability in the backward problem in time,  the uniqueness in determining an initial value, the uniqueness of solution by the decay rate as $t \to \infty$ and stability in an inverse source problem of determining $t$ - dependent factor in the source by observation at one point over $(0, T)$. The authors study an inverse source problem for a fractional diffusion equation in \cite{yz1}. Under the assumption that the unknown source term is time-independent, an analytical solution can be deduced based on the method of the eigenfunction expansion. Then, the uniqueness of the inverse problem is proved by analytic continuation and Laplace transform. The paper \cite{ym5} deals with an inverse problem of simultaneously identifying the space-dependent diffusion coefficient and the fractional order in the one dimensional time-fractional diffusion equation with smooth initial functions by using boundary measurements. The uniqueness results for the inverse problem are proved on the basis of the inverse eigenvalue problem, and the Lipschitz continuity of the solution operator is established. In \cite{feng}, the authors study an inverse random source problem for the time-fractional diffusion equation, where the source is driven by a fractional Brownian motion. Given the random source, the direct problem is to study the stochastic time-fractional diffusion equation. The inverse problem is to determine the statistical properties of the source from the expectation and variance of the final time data. For the direct problem, the authors show that it is well-posed and has a unique mild solution under a certain conditions. For the inverse problem, the uniqueness is proved and the instability is characterized. The authors consider an inverse boundary value problem for diffusion equations with multiple fractional time derivatives  and they prove the uniqueness in determining the number of fractional time-derivative terms, the orders of the derivatives and spatially varying coefficients in \cite{oleg}.  However, there are  only a few papers involving both fractional Laplacian and fractional time derivatives. For instance, in \cite{salih1}, the authors study a nonlocal inverse problem related to the space-time fractional equation. The existence of the solution for the inverse problem is proved by using the quasi-solution method which is based on minimizing an error functional between the output data and the additional data. In this context, an input-output mapping is defined and continuity of the mapping is established. The uniqueness of the solution for the inverse problem is also proved by using eigenfunction expansion of the solution and some basic properties of fractional Laplacian. A numerical method based on discretization of the minimization problem, steepest descent method, and least squares approach is proposed for the solution of the inverse problem. A nonlocal inverse source problem for a one-dimensional space-time fractional diffusion equation is studied in \cite{salih2}. At first, they define and analyze the direct problem for the space-time fractional diffusion equation. Later, they define the inverse source problem. Furthermore, they set up an operator equation  and derive the relation between the solutions of the operator equation and the inverse source problem. They also prove some important properties of the operator. By using these properties and the analytic Fredholm theorem, it is proved that the inverse source problem is well posed, i.e.  the solution can be determined uniquely and depends continuously on additional data. In \cite{ulusoy1}, they consider a nonlocal inverse problem and show that the fractional exponents $\beta$, $\alpha$ and $\gamma$, where $\beta$ is the order of the time-fractional derivative and $\alpha$ and $\gamma$ are exponents of fractional Laplacian operator,  are determined uniquely by the data $u(x, T) = h(x), 0 \le t \leq T$. The existence of the solution for the inverse problem is proved using the quasi-solution method which is based on minimizing an error functional between the output data and the additional data. In this context, an input-output mapping is defined and its continuity is established. The uniqueness of the solution for the inverse problem is proved by means of eigenfunction expansion of the solution to the forward problem and some basic properties of fractional Laplacian. In \cite{ng1}, the authors study a diffusion equation of the Kirchhoff type with a conformable fractional derivative. The global existence and uniqueness of mild solutions are established. Some regularity results for the mild solution are also derived. This study can be regarded as a continuation of the series of works mentioned above on fractional inverse problems. \\

This paper is organized as follows: In section \ref{pre}, we present some preliminaries used throughout the paper. In section \ref{inv_pbm}, we formulate the direct and inverse problems, we prove that the considered inverse problem has a unique solution and the inverse problem is ill-posed,  we also reformulated the inverse problem  as a minimization problem. Section \ref{rec} is concerned with the proposed reconstruction approach. Some numerical simulations are presented in Section \ref{num}. The conclusions and possible directions on the problem are given in Section \ref{con}.

\section{Preliminaries}\label{pre}
In this section, we set some basic notations and recall some definitions and theorems.\\

\noindent By $L^2(\Omega)$, we denote the usual $L^2$-space with the inner product $(\, ,\, )$ and by $H_0^1(\Omega)$, $H^1(\Omega)$, etc we denote the usual Sobolev spaces. By $H^\alpha(0, T)$, we denote the fractional Sobolev space in time (see Adams \cite{adams1975sobolev}). Especially, for $s\in (0,1)$, the fractional Sobolev space $H^s(\Omega)$ is defined by 
$$
H^s(\Omega)=\left\{u \in L^2(\Omega):|u|_{H^s(\Omega)}:=\left(\iint_{\Omega^2} \frac{|u(x)-u(y)|^2}{|x-y|^{d+2 s}} d x d y\right)^{\frac{1}{2}}<\infty\right\}.
$$
Its natural norm is defined by
$$\|u\|_{H^s(\Omega)}:=\left(\| u\|_{L^2(\Omega)}^2+|u|_{H^s(\Omega)}^2\right)^{1/2}.$$
Moreover, we define the fractional space $\tilde{H}^{s}(\Omega)$ of order $s \in (0,1)$ as follows: 
$$
\widetilde{H}^{s}(\Omega):=\left\{u \in H^{s}\left(\mathbb{R}^{d}\right): \operatorname{supp} u \subset \overline{\Omega}\right\}.
$$
It may also be defined through interpolation as follows:
$$\widetilde{H}^{s}(\Omega):=\left[L^2(\Omega), H_0^1(\Omega) \right]_s.$$
The bilinear form associated to the space $\widetilde{H}^{s}(\Omega)$ is given by
$$
\mathcal{B}_s(u,v):=C_{d, s} \iint_{\left(\mathbb{R}^{d} \times \mathbb{R}^{d}\right) \backslash\left(\Omega^{c} \times \Omega^{c}\right)} \frac{(u(x)-u(y))(v(x)-v(y))}{|x-y|^{d+2 s}}  d x \, d y.
$$
\begin{proposition}[see \cite{dipierro2017nonlocal,du2013nonlocal}] Let $u, v : \mathbb{R}^{d} \longrightarrow \mathbb{R}$ be smooth functions, then
\begin{equation*} \label{integration_laplace}
\int_{\Omega} v(x)(-\Delta)^{s} u(x) d x  = \frac{\mathcal{B}_s(u,v)}{2} -\int_{\Omega^{c}} v(x) \mathcal{N}_{s} u(x) d x,
\end{equation*}
where $\mathcal{N}_{s}$ denotes the non-local Neumann operator associated to $(-\Delta)^{s}$ and defined as
$$
\mathcal{N}_{s} v(x):=C_{d, s} \int_{\Omega} \frac{v(x)-v(y)}{|x-y|^{d+2 s}} d y.
$$
The following proposition is concerned with a fractional integration by parts formula. It gives the relationship between  Caputo and Riemann-Liouville fractional derivatives.
\end{proposition}
\begin{proposition}[see \cite{agrawal2007fractional,almeida2011necessary}] \label{integration1} Let $\alpha \in (0,1)$. Let $\mu_{1}$ and $\mu_{2}$ be two absolutely integrable functions. Then, we have
\begin{equation*}\label{integ1}
\int_{0}^{T} \mu_{2}(t) \partial_{t}^{\alpha} \mu_{1}(t) d t=\int_{0}^{T} \mu_{1}(t) D_{t}^{\alpha} \mu_{2}(t) d t+\left[\mu_{1}(t) J_{T^{-}}^{1-\alpha} \mu_{2}(t)\right]_{t=0}^{t=T},
\end{equation*}
where $J_{T^{-}}^{1-\alpha} \mu_{2}(t)$ denotes the right Riemann-Liouville fractional integral of $\mu_2(t)$ defined by
$$
J_{T^{-}}^{1-\alpha} \mu_2(t)=\frac{1}{\Gamma(1-\alpha)} \int_{t}^{T} \frac{\mu_2(\tau)}{(\tau-t)^{\alpha}} \mathrm{d} \tau,
$$
and $D_{t}^{\alpha} \mu_{2}(t)$ denotes the backward Riemann-Liouville fractional derivative  defined by
\begin{equation*}\label{riemann}
D_{t}^{\alpha} \mu_2(t):=-\left(J_{T^{-}}^{1-\alpha} \mu_2(t)\right)^{\prime}=\frac{-1}{\Gamma(1-\alpha)} \frac{d}{d t} \int_{t}^{T} \frac{\mu_2^{\prime}(\tau)}{(\tau-t)^{\alpha}} \mathrm{d} \tau.
\end{equation*}
\end{proposition}
\noindent Let us introduce the Mittag-Leffler function which is defined on the complex set $\mathbb{C}$ by 
$$
E_{\alpha, \beta}(z):=\sum_{k=0}^{\infty} \frac{z^{k}}{\Gamma(\alpha k+\beta)}, \quad z \in \mathbb{C},
$$
where $\alpha>0$ and $\beta \in \mathbb{R}$ are arbitrary constants. Notice that this complex function depends on two parameters. In particular, it
generalizes the exponentials in view of the identity $E_{1,1}(z)=e^z$ for all $z\in \mathbb{C}$. Moreover, it plays a central role in fractional diffusion equations.  The
following results of this family of functions are useful to derive the solution representation of the direct problem \eqref{space_time}. 
\begin{lemma}\cite[Theorem 4.3]{diethelm2010analysis}\label{mittag}
Let $\alpha>0$ and $\lambda>0$, then we have
$$
\partial_{t}^{\alpha} E_{\alpha, 1}\left(-\lambda t^{\alpha}\right)=-\lambda E_{\alpha, 1}\left(-\lambda t^{\alpha}\right), \quad t>0.
$$
Moreover, the following identity holds for integer-order differentiation:
$$
\frac{d}{dt} E_{\alpha, 1}\left(-\lambda t^{\alpha}\right)=-\lambda\, t^{\alpha-1} E_{\alpha, \alpha}\left(-\lambda t^{\alpha}\right), \quad t>0.
$$
\end{lemma}
\begin{lemma}\cite[Theorem 1.6]{podlubny1999fractional}\label{mittag2} Let $0<\alpha<1$ and   $\pi \alpha / 2<\mu<\pi \alpha$. Then there exists a constant $C_{0}=C_{0}(\alpha, \mu)>0$ such that
$$
\left|E_{\alpha, 1}(z)\right| \leqslant \frac{C_{0}}{1+|z|}, \quad \mu \leqslant|\arg (z)| \leqslant \pi.
$$
\end{lemma}
\section{The direct and the inverse problems}\label{inv_pbm} 
In this section, we formulate the direct and inverse problems, then we prove that the considered inverse problem has a unique solution. We also prove that the inverse problem is ill-posed. First of all, we need to define a solution formula for the direct problem \eqref{space_time}. By using the eigenfunction expansion method, following  \cite{acosta2019finite},  \cite{nane}, \cite{luchko}, \cite{y4}, we get the following useful formula for the weak solution of the direct problem \eqref{space_time}:  

\begin{equation}\label{expansion}
\small
  u(x,t):= \sum_{k=1}^{\infty}g_k\,E_{\alpha, 1}\left(-\lambda_k t^{\alpha}\right) \, \varphi_k(x), \quad \forall (x,t) \in \Omega \times (0,T),  
\end{equation}
where $g_k=(g,\varphi_k)$ and the family $\{(\varphi_k, \lambda_k)\}_{k\geq 1}$ represents the eigenpairs of the fractional Laplace operator $(-\Delta)^s$ on $\Omega$ with homogeneous Dirichlet condition, i.e.
\begin{equation*}
\left\{\begin{aligned}
 (-\Delta)^s \varphi_k & = \lambda_k \varphi_k & & \text { in } \Omega, \\
\varphi_k & = 0 & & \text { in }  \Omega^c.
\end{aligned}\right.
\end{equation*}
In addition, it is well-known that the fractional Laplacian operator has a sequence of eigenvalues satisfying
$$
0<\lambda_1 \leq \lambda_2 \leq \cdots \leq \lambda_k \leq \cdots \text { and } \lim _{k \rightarrow+\infty} \lambda_k=+\infty.
$$
Besides, the set of eigenfunctions $\left\{\varphi_k\right\}_{k=1}^{\infty}$ forms an orthonormal basis of $L^2(\Omega)$.

\begin{remark} Unlike the classical Laplacian, it is proved in \cite{grubb2015spectral,ros2014local} that eigenfunctions of the fractional Laplacian
are in general non-smooth. More precisely, for each $k$, $k =  1, 2, 3, \cdots $,  the eigenfunction $\varphi_k$ belongs to $H^{s+1 / 2-\epsilon}\left(\mathbb{R}^d\right)$ where  $\epsilon>0$ is an arbitrary small real number.
\end{remark}
Following the same idea given in \cite{y4,acosta2019finite} and based on the formula \eqref{expansion} and Lemmas \ref{mittag} and \ref{mittag2}, we have the following theorem that states the regularity result for the solution of the problem \eqref{space_time}:
\begin{theorem} Let $(\alpha, s) \in (0,1)^2$, and $g \in L^{2}(\Omega)$ be given. Then, problem \eqref{space_time} admits a unique weak solution
$u \in \mathcal{H}^{\alpha,s}(0,T;\Omega)$. Furthermore, there exists a constant $C>0$ such that
\begin{equation*}\label{est0}
\left\|\partial_{t}^{\alpha} u\right\|_{C([0, T] ; L^2(\Omega))}\leq C\|g\|_{L^{2}(\Omega)},
\end{equation*}
and
\begin{equation*}\label{est1}
\left\|\partial_{t}^{\alpha} u(.,t)\right\|_{L^2(\Omega)}+\left\|u(.,t)\right\|_{H^{s+\gamma}(\Omega)}\leq C\, t^{-\alpha}\,\|g\|_{L^{2}(\Omega)},
\end{equation*}
\end{theorem}
where
$$\mathcal{H}^{\alpha,s}(0,T;\Omega):=\Big\{v \in  C([0, T] ; L^2(\Omega)) \cap C((0, T] ; \widetilde{H}^{s}(\Omega) \cap H^{s+\gamma}(\Omega))\; \text{such that}\;\partial_t^\alpha v \in C((0,T]; L^2(\Omega))\Big\},$$
$\gamma:=\min \{s, 1 / 2-\varepsilon\}$ and $\varepsilon>0$ arbitrarily small. \\

The inverse problem that we consider consists of identifying the initial value $g$, from noise measurement of the final time solution. More precisely, the considered inverse problem can be formulated as finding $g^\star \in L^2(\Omega)$ and $u^\star \in \mathcal{H}^{\alpha,s}(0,T;\Omega)$ in the following problem:
\begin{equation}\label{inv}
\left\{\begin{aligned}
\partial_t^\alpha u^\star = - (-\Delta)^s u^\star, & & & (x, t) \in  \Omega \times(0, T), \\
u^\star & = 0,& & (x, t) \in  \Omega^c \times(0, T), \\
u^\star & = g^\star ,& & (x, t) \in  \Omega \times\{0\},\\
u^\star & = h,& & (x, t) \in  \Omega \times\{T\}.
\end{aligned}\right.
\end{equation}
The following proposition is concerned with the series representation of the solution of problem \eqref{inv} using the eigenfunction expansion method.
\begin{proposition}\label{representation} Let $h \in L^2(\Omega)$ be a given function. Then, the solution of the problem \eqref{inv} can be represented as follows
\begin{equation*}\label{sol_inv}
g^\star(x)= \sum_{k=1}^{\infty} \frac{\big( h, \varphi_k\big) }{ E_{\alpha,1}(-\lambda_k T^{\alpha})}\, \varphi_k(x),
\end{equation*}
and
\begin{equation*}\label{over_inv}
u^\star(x,t)= \sum_{k=1}^{\infty}\frac{\big( h, \varphi_k\big)}{ E_{\alpha,1}(-\lambda_k T^{\alpha})}\,E_{\alpha,1}(-\lambda_k t^{\alpha}) \,\varphi_k(x).
\end{equation*}
\end{proposition}
\begin{proof}
From \eqref{expansion}, the solution of 
    \begin{equation*}
\left\{\begin{aligned}
\partial_t^\alpha u^\star = - (-\Delta)^s u^\star,& & & (x, t) \in  \Omega \times(0, T), \\
u^\star & = 0,& & (x, t) \in  \Omega^c \times(0, T), \\
u^\star & = g^\star, & & (x, t) \in  \Omega \times\{0\},
\end{aligned}\right.
\end{equation*}
is given by 
\begin{equation}\label{ouver_inv0}
u^\star(x,t):= \sum_{k=1}^{\infty}(g^\star, \varphi_k)\,E_{\alpha, 1}\left(-\lambda_k t^{\alpha}\right)\, \varphi_k(x).
\end{equation}
By taking $t=T$ and using the fact that $u^\star(.,T) =h$ in $\Omega$, we get
\begin{equation*}
h= \sum_{k=1}^{\infty}(g^\star, \varphi_k)\,E_{\alpha, 1}\left(-\lambda_k T^{\alpha}\right)\, \varphi_k(x).
\end{equation*}
Multiplying both sides of the above equation by $\varphi_k$ and integrating with respect to $x$, we get 
$$\big( h, \varphi_k\big) = \big( g^\star, \varphi_k\big)\,E_{\alpha,1}(-\lambda_k T^{\alpha}).$$
Since $E_{\alpha,1}(-\lambda_kT^{\alpha})>0$, it follows:
\begin{equation}\label{sssol_inv}
\big( g^\star, \varphi_k\big) = \frac{\big( h, \varphi_k\big)}{E_{\alpha,1}(-\lambda_k T^{\alpha})}.
\end{equation}
Hence $g^\star$ is given by
\begin{equation*}
g^\star(x)= \sum_{k=1}^{\infty} \frac{\big( h, \varphi_k\big) }{ E_{\alpha,1}(-\lambda_k T^{\alpha})}\, \varphi_k(x).
\end{equation*}
Substituting \eqref{sssol_inv} into \eqref{ouver_inv0} yields 
\begin{equation*}
u^\star(x,t)= \sum_{k=1}^{\infty}\frac{\big( h, \varphi_k\big) }{ E_{\alpha,1}(-\lambda_k T^{\alpha})}\,E_{\alpha,1}(-\lambda_k t^{\alpha}) \,\varphi_k(x).
\end{equation*}
The proof is complete. 
\end{proof}
Now we prove a uniqueness theorem. Let  $u_1$ and $u_2$ be the solutions of the problem \eqref{space_time} correspond to the initial values $g_j \in L^2(\Omega)$, $j=1,2$, respectively.
\begin{theorem} Let $u_1$ and $u_2$ be the solutions of the problem \eqref{space_time}. If $u_1(x, T) = u_2(x, T), \, x \in \Omega$, then we have 
$$g_1=g_2\; \text{in}\; \Omega.$$
\end{theorem}
\begin{proof}
    Using the series representations of solutions $u_1$ and $u_2$, we have 
\begin{equation}
u_j(x, T)=\sum_{k=1}^{\infty}\left\{\left(g_{j}, \varphi_k\right) E_{\alpha, 1}\left(-\lambda_k T^\alpha\right)\right\} \varphi_k(x),\; \; \text{for} \; j=1,2.
\end{equation}
As $u_1(.,T)=u_2(.,T)$ in $\Omega$, we deduce that
\begin{equation*}
\sum_{k=1}^{\infty}\left\{\left(g_{1}, \varphi_k\right) E_{\alpha, 1}\left(-\lambda_k T^\alpha\right)\right\} \varphi_k(x)=\sum_{k=1}^{\infty}\left\{\left(g_{2}, \varphi_k\right) E_{\alpha, 1}\left(-\lambda_k T^\alpha\right)\right\} \varphi_k(x).
\end{equation*}
First, we multiply both sides of the above equality by $\varphi_k(x)$ and integrate the resulting equation with respect to $x$. Then by using the fact that $\varphi_k=0$ in $\Omega^c$, we obtain
    $$
\left(g_{1}, \varphi_k\right) E_{\alpha, 1}\left(-\lambda_k T^\alpha\right) = \left(g_{2}, \varphi_k\right) E_{\alpha, 1}\left(-\lambda_k T^\alpha\right), \; \text{for all}\; k\geq 1.
$$
It follows:
$$
\left(g_{1}-g_{2}, \varphi_k\right) E_{\alpha, 1}\left(-\lambda_k T^\alpha\right) =0, \; \text{for all}\; k\geq 1.
$$
Since $E_{\alpha, 1}\left(-\lambda_k t^\alpha\right)>0$, $t>0$ is completely monotonic (see \cite{gorenflo1997fractional}),  $E_{\alpha, 1}\left(-\lambda_k T^\alpha\right)>0$, we have the following: 
$$
\left(g_{1}, \varphi_k\right)  =\left(g_{2}, \varphi_k\right), \; \text{for all}\; k\geq 1,
$$
implies $g_{1}=g_{2}$ in $\Omega$. Thus the proof is complete.
\end{proof}

To find an estimate of $g$ from the additional data $h$, the most common method is to minimize the discrepancy
$$\left\|\mathcal{A}(g)-h\right\|_{L^2(\Omega)}^2,$$
where the operator $\mathcal{A}$ is defined by 
\begin{equation}\label{operator}
\begin{aligned}
\mathcal{A}: L^{2}(\Omega) & \longrightarrow L^{2}(\Omega), \\
g & \longmapsto u(.,T).
\end{aligned}
\end{equation}
We know that the linear operator $\mathcal{A}$ is self-adjoint, for more details we refer the readers to \cite{karapinar2020identifying,wang2015quasi}. Next, we prove that the operator $\mathcal{A}$ is compact. 
\begin{proposition}
   The linear operator $\mathcal{A}$ defined by  \eqref{operator} is a compact operator from $L^2(\Omega)$ into $L^2(\Omega)$.
\end{proposition}
\begin{proof}
By \eqref{expansion} and \eqref{operator}, we deduce that
\begin{equation}\label{inifinte}
    \mathcal{A}(g)=\sum_{k=1}^{+\infty} g_k\,E_{\alpha, 1}\left(-\lambda_k T^{\alpha}\right)\, \varphi_k(x),\; \forall g \in L^2(\Omega).
\end{equation}
We define the finite rank operators $\mathcal{A}_N$ as follows:
\begin{equation}\label{finite}
\mathcal{A}_N(g) :=\sum_{k=1}^{N} g_k\,E_{\alpha, 1}\left(-\lambda_k T^{\alpha}\right)\, \varphi_k(x),\; \forall g \in L^2(\Omega).
\end{equation}
From \eqref{inifinte} and \eqref{finite}, we get
$$\| \mathcal{A}(g)-\mathcal{A}_N(g)\|_{L^2(\Omega)}^2= \sum_{k=N+1}^{+\infty}\big|E_{\alpha, 1}\left(-\lambda_k T^{\alpha}\right)\big|^2\, |g_k|^2.$$
Thanks to Lemma \ref{mittag2}, we obtain 
$$\| \mathcal{A}(g)-\mathcal{A}_N(g)\|_{L^2(\Omega)}\leqslant \frac{C_0}{T^\alpha\, \lambda_{N}}\| g\|_{L^2(\Omega)}.$$
Therefore, $\| \mathcal{A}(g)-\mathcal{A}_N(g)\|_{L^2(\Omega)} \longrightarrow 0$ in the sense of operator norm in $L(L^2(\Omega); L^2(\Omega))$ as $N \longrightarrow \infty.$
\end{proof}
Since equations with compact operators are  ill-posed, the inverse  problem under consideration is also ill-posed.\\

It is well known that in practical applications, the given data $h$ is typically not exact, but rather a distortion of the unknown  $u[g_{exact}]$. This distortion is often modeled by an additive noise or an error term $\theta$. Denote by $h^\theta$ a noisy function of $h$ satisfying 
$$\|h-h^\theta \| \leqslant \theta.$$
Due to this noise associated with the measured data, the solution becomes very sensitive to the
measured data, which causes severe numerical instabilities. Thus the considered inverse problem is ill-posed in the sense of Hadamard \cite{hadamard1902problemes}. In the following, we consider an example 
to see that the considered inverse problem is ill-posed. More precisely, we will prove that the solution to the inverse problem does not depend continuously on the final time data $h$.
In doing so, let us choose an input final data $h_1$ as 
$$h_p(x)=\frac{\varphi_p(x)}{\sqrt{\lambda_p}}.$$
From Proposition \ref{representation}, the initial data corresponding to $h^p$ is represented as follows:
\begin{equation*}
g^p(x)= \sum_{k=1}^{\infty} \frac{\big( h_1, \varphi_k\big) }{ E_{\alpha,1}(-\lambda_k T^{\alpha})}\, \varphi_k(x).
\end{equation*}
Using the fact the set of eigenfunctions  $\left\{\varphi_{k}\right\}_{k=1}^{\infty}$ forms a complete orthonormal basis of $L^{2}(\Omega)$, we have
\begin{equation*}
g^p(x)= \frac{\varphi_p(x)}{\sqrt{\lambda_p}\,E_{\alpha,1}(-\lambda_p T^{\alpha})}.
\end{equation*}
On the other hand, let us choose other input final data as $h_2 \equiv 0$. By Proposition \ref{representation}, the initial value corresponding to $h_0$ is $g \equiv 0$. An error in $L^2$ norm
between two input final data is
$$\| h_p-h_0\|_{L^2(\Omega)}= \Big\| \frac{\varphi_p}{\sqrt{\lambda_p}}\Big\|_{L^2(\Omega)}.$$
Since the family $\left\{\varphi_k\right\}_{k=1}^{\infty}$ forms an orthonormal basis of $L^2(\Omega)$, one can get 
$$\| h_p-h_0\|_{L^2(\Omega)}= \frac{1}{\sqrt{\lambda_p}}.$$
It follows
\begin{equation}\label{data}
\lim_{p\to\infty} \| h_p-h_0\|_{L^2(\Omega)}  =0.    
\end{equation}
The error in the $L^2$ norm between the two corresponding initial values is
$$\| g^p-g\|_{L^2(\Omega)}=\Big\| \frac{\varphi_p}{\sqrt{\lambda_p}\,E_{\alpha,1}(-\lambda_p T^{\alpha})}\Big\|_{L^2(\Omega)}= \frac{1}{\sqrt{\lambda_p}\,E_{\alpha,1}(-\lambda_p T^{\alpha})}.$$
From Lemma \ref{mittag2}, we deduce that there exists $C=C(T,\alpha,C_0)>0$ such that
$$\| g^p-g\|_{L^2(\Omega)}\geq C\, \sqrt{\lambda_p} \,.$$
This leads to \begin{equation}\label{solution}
\lim_{p\to\infty} \| g^p-g\|_{L^2(\Omega)}  =+\infty.    
\end{equation} Combining \eqref{data} and \eqref{solution}, we conclude that the inverse problem that we consider is ill-posed. \\

In order to handle the possible numerical instability of the inverse problem, there are some regularization methods in the literature. For instance, the quasi-reversibility method \cite{lattes1969method}, an alternate iterative method \cite{kozlov1989iterative}, and the quasi-boundary value method \cite{abdulkerimov1977regularizatio}. In the current work, we employ one of the most commonly used methods for the regularization of ill-posed problems, that is, the Tikhonov regularization method. We define the Tikhonov regularization
functional as follows:
$$\mathcal{T}_\nu(g):=\frac{1}{2}\left\|\mathcal{A}(g)-h^\theta\right\|_{L^2(\Omega)}^2+\frac{\gamma}{2} \left\|g\right\|_{L^2(\Omega)}^2,$$
where $\gamma>0$ is a positive constant, called the regularization parameter. In the function above, the first term denotes the defect between the exact data and the
noisy data, and the second term is a penalty term for stabilizing the numerical solution. Consequently, the considered inverse  problem may be reformulated and modeled by the following regularized optimization problem:
\begin{equation*}
(\mathcal{P}_{op}) \left\{\begin{array}{l}
\text { Find } g^\star \in L^2(\Omega)\text { such that } \\
\\
\mathcal{T}_\nu(g^\star):=\displaystyle\min _{g \in L^{2}(\Omega)} \mathcal{T}_\nu(g).
\end{array}\right.
\end{equation*}
We know that for problem $(\mathcal{P}_{op})$,  there exists a unique minimizer $g_\gamma^\theta$ called Tikhonov regularized solution
which converges to the exact solution $g^\star$ under a suitable choice of the regularization parameter $\gamma$, see \cite{engl1996regularization}. For more details about the analysis of the optimization problem, interested readers can follow the same technique developed in \cite{abdelwahed2022inverse,bensaleh2021inverse,bensalah2021inverse,jiang2020numerical, hrizi2022determination}. The next section is concerned with a numerical method for finding the unique minimizer of the Tikhonov regularization functional $\mathcal T_\nu(g).$

\section{Reconstruction approach}\label{rec} This section is devoted to  numerical reconstruction approach for solving the minimization problem $(\mathcal{P}_{op})$. The proposed approach is based on two steps:
\begin{itemize}
    \item The first one is concerned with the derivation of an optimality condition that provides a new characterization of the unknown term $g^\star$.
    \item The second one is that  we employ the conjugate gradient algorithm to solve the variational problem.
\end{itemize}
Now we derive a first-order optimality condition that provides a simplified characterization of the unknown initial value $g^\star$. The determination of this condition is based on the calculation of the gradient of  $\mathcal{T}_\nu$, which can be obtained by constructing an adjoint problem.\\

Hereafter, we denote by $u_g$ the solution of \eqref{space_time} to emphasize its dependency upon the
unknown function $g$. We point out that the weak formulation of problem \eqref{space_time} reads as follows: Find $u_g \in \mathcal{H}^{\alpha,s}(0, T; \Omega)$ such that $u_g(.,0)=0$ and
\begin{equation}\label{weak_sequel} 
\int_0^T \int_\Omega \partial_t^\alpha u_g\, w \, dx\,dt + \int_0^T \mathcal{B}_s(u_g(.,t), w(.,t))\,dt=0,
\end{equation}
for any test function $w \in H^\alpha(0,T; \widetilde{H}^{s}(\Omega) \cap H^{s+\gamma}(\Omega))$ with $J_{T^-}^{1-\alpha}w=0$  in $\Omega \times \{T\}$. 
\begin{remark}\label{frechet} The map $g \longmapsto u_g$ is differentiable in the sense of Fréchet and the linearity of \eqref{space_time}
immediately yields 
$$
u_g^{\prime}\cdot p=\lim _{\epsilon \rightarrow 0} \frac{u_{g+\epsilon p}-u_g}{\epsilon}=u_{p},\; \; \forall p \in L^{2}(\Omega),
$$
here $u_g^{\prime}\cdot p$ denotes the Fréchet derivative of $u_g$ in the direction $p$ and $u_p$ is the solution problem of \eqref{space_time} with $g=p$. Notice that, from Proposition \ref{integration1} and the identity \eqref{weak_sequel}, one can check that the function $u_p$ satisfies 
\begin{equation}\label{weak_forward}  \int_0^T \int_\Omega  u_p\, D_t^\alpha w \, dx\,dt + \int_0^T \mathcal{B}_s(u_p(.,t), w(.,t))\,dt=\int_\Omega p(x)\, J_{T^-}^{1-\alpha}w(x,0)\, dx,
\end{equation}
for any test function $w \in H^\alpha(0,T; \widetilde{H}^{s}(\Omega) \cap H^{s+\gamma}(\Omega))$ with $J_{T^-}^{1-\alpha}w=0$  in $\Omega \times \{T\}$. 
\end{remark}
In order to establish the optimality condition, we need the Fréchet derivative $\mathcal{T}_\nu^{\,\prime}(g)$ of the objective functional $\mathcal{T}_\nu(g)$. By sample calculations, one can easily derive the following: 
\begin{equation} \label{der1}
\begin{split}
\mathcal{T}_\nu^{\,\prime}(g) \cdot p & =\lim _{\epsilon \rightarrow 0} \frac{\mathcal{T}_\nu(g+\epsilon p)-\mathcal{T}_\nu(g)}{\epsilon} \\
 & = \int_{\Omega}
\big[u_g(x,T)-h^\theta(x) \big]\;u_p(x,T) \,\dx + \gamma
\int_{\Omega}g(x)\,p(x)\,\dx.
\end{split}
\end{equation}
In order to reduce the computational costs for the Fréchet derivatives, we state it in the natural form. Namely, we need to find an explicit function $R(x)$ such that $\mathcal{T}_\nu^{\,\prime}(g) \cdot p=(R, p)$. So, we need to replace the term $\int_{\Omega}
\big[u_g(x, T)-h^\theta(x) \big]\;u_p(x, T) \,\dx$ in \eqref{der1} by $p$ times a function of $x$. Therefore, we introduce the following adjoint problem:
\begin{eqnarray}\label{addjoint-z}
\left\{
  \begin{array}{rllcl}
    \partial_{T^-}^\alpha z_g = - (-\Delta)^s z_g & + & (u_g(.,T)-h^\theta)\,\delta(t-T),    &(x, t) \in  & \Omega \times (0,T], \\
     z_g&=&0,  & (x, t) \in  &  \Omega^c\times(0,T], \\
    z_g&=&0,  & (x, t) \in & \Omega \times \{T\},
  \end{array}
\right.
\end{eqnarray}
where $\delta(t-T)$ is the Dirac delta function at the time $t=T$. The weak formulation of the adjoint problem \eqref{addjoint-z} reads as: Find $z_g \in H^\alpha(0,T; \widetilde{H}^{s}(\Omega) \cap H^{s+\gamma}(\Omega))$ such that $J_{T^-}^{1-\alpha}z_g=0$ in $\Omega \times \{T\}$ and
\begin{equation}\label{weak_adjoint}
 \int_0^T \int_\Omega D_t^\alpha z_g\, w \, dx\,dt + \int_0^T \mathcal{B}_s(z_g(.,t), w(.,t))\,dt=\int_\Omega (u_g(x,T)-h^\theta)\, w(x,T) \, dx,   
\end{equation}
for any test function $w \in W^{\alpha,s}(0, T; \Omega)$ with $w(.0)=0$  in $\Omega$. After these considerations, we can take $z_g$ and $u_p$ as mutual test functions in identities \eqref{weak_forward} and \eqref{weak_adjoint} and we get
\begin{equation}\label{reduced}
 \int_{\Omega}
\big[u_g(x, T)-h^\theta(x) \big]\;u_p(x, T) \,\dx:=\int_\Omega p(x)\,J_{T^-}^{1-\alpha}z_g(x,0)\, dx.   
\end{equation}
Therefore \eqref{reduced} in \eqref{der1} yields
\begin{equation*} \label{der2}
\mathcal{T}_\nu^{\,\prime}(g) \cdot p = \int_\Omega J_{T^-}^{1-\alpha}z_g(x,0)\,p(x)\, dx + \gamma
\int_{\Omega}g(x)\,p(x)\,\dx.
\end{equation*}
Based on the above identity, we deduce that the solution to the minimization problem $(\mathcal{P}_{op})$ satisfies the
following optimality condition
$$J_{T^-}^{1-\alpha}z_{g^\star}(.,0)+ \gamma
g^\star=0, \; \text{ in }\; \Omega.$$
Next we propose a numerical algorithm  for identifying the minimizer of function $\mathcal{T}_\nu(g)$ from noisy measurement of the final time. The numerical algorithm that we propose is based on the conjugate gradient method and Morozov's discrepancy principle (see, e.g., \cite{Morozov1}). Let $g_k$ be the $k$th approximate
solution to $g(x)$. Denote
\begin{equation}\label{process_algo}
g_{k+1}=g_{k}+\zeta_{k}\, d_{k}, \quad k=0,1,2, \cdots,
\end{equation}
where the initial guess $g_0$ is given, the term $\zeta_{k}$ is the step size, and $d_{k}$ is a descent direction in the $k$th iteration. The conjugate gradient method uses the following iteration formula to update the descent direction:
\begin{eqnarray}\label{d_algo}
d_{k}=\left\{
  \begin{array}{rllcl}
    &-\mathcal{T}_\nu^{\, \prime}(g_0)   &\hbox{if} & k=0, \\
    \\
     &-\mathcal{T}_\nu^{\, \prime}(g_k)+\alpha_{k}\,d_{k-1}  & \hbox{if}& k\geq 1,
  \end{array}
\right.
\end{eqnarray}
where $\alpha_k$ is the conjugate coefficient calculated by
\begin{eqnarray}\label{alpha_algo}
\alpha_{k}=\left\{
  \begin{array}{rllcl}
    &0   &\hbox{if} & k=0, \\
     &\frac{\ds \int_{\Omega}\big|\mathcal{T}_\nu^{\, \prime}(g_k)\big|^{2} \mathrm{~d} x}{\ds \int_{\Omega}\big|\mathcal{T}_\nu^{\, \prime}(g_{k-1})\big|^{2} \mathrm{~d} x}  & \hbox{if}& k\geq 1.
  \end{array}
\right.
\end{eqnarray}
Since the problem \eqref{space_time} is linear with respect to the initial value $g$, one can deduce from \eqref{der1} that 
\begin{equation*} 
\mathcal{T}_\nu(g_{k}+\zeta_{k}\, d_{k}) = \frac{1}{2}\int_{\Omega}
\big(u_{g_k}(x,T)+\zeta_k\, u_{d_k}(x,T)-h^\theta(x) \big)^2\,\dx + \frac{\gamma}{2}
\int_{\Omega}\big( g_{k}+\zeta_{k}\, d_{k}\big)^2\,\dx.
\end{equation*}
To this end, we determine the step size $\zeta_k$ by imposing the following condition $$\displaystyle\frac{\partial \mathcal{T}_\nu }{\partial \zeta}(g_{k}+\zeta_k d_{k})=0.$$
Therefore, 
\begin{equation}\label{zeta_algo}
\zeta_{k}=-\frac{\ds \int_{\Omega}\big(u_{g_k}(x, T)-h^\theta\big)\,u_{d_k}(x,T) \, \dx+\gamma \int_{\Omega} g_{k}\, d_{k} \,\dx}{\ds \int_{\Omega} u_{d_k}^{2}(x,T) \,\dx+\gamma \int_{\Omega} d_{k}^{2} \,\dx}.
\end{equation}
We can summarize the main steps of our reconstruction approach in the following algorithm: 

\begin{algorithm}[H]
\begin{enumerate}
\item Initialize $g_0$ andset $k=0.$

\item Solve the direct problem \eqref{space_time} with $g=g_{k}$, and compute  the residual $$r_{k}=u_{g_k}(x,T)-h^\theta(x),\, x\in \Omega.$$

\item Solve the adjoint problem \eqref{addjoint-z} and evaluate the gradient $\mathcal{T}_\nu^{\, \prime}(g_k)$.

\item Calculate the conjugate coefficient $\alpha_{k}$ by \eqref{alpha_algo} and the direction $d_{k}$ by \eqref{d_algo}.

\item Compute $u_{d_k}$ via solving the  problem \eqref{space_time} with $g=d_{k}$.

\item Calculate the step size $\zeta_{k}$ by \eqref{zeta_algo}.

\item Update the initial value $g_{k+1}$ by \eqref{process_algo}.

\item Set $k = k +1$ and go to Step (2), repeat the process until a stopping condition is satisfied.

\end{enumerate}
    \caption{\it Conjugate Gradient Method {\bf (CGM)}}
    \label{algo1}
\end{algorithm}
Notice that the most important point is to find a suitable stopping rule for an iteration procedure. To deal with this issue, we use the well-known Morozov's discrepancy principle \cite{morozov2012methods}. It is shown that 
\begin{equation}\label{con_cond}
\mathcal{R}_{k} \leqslant \sigma \theta < \mathcal{R}_{k-1},\quad \text{with} \quad \mathcal{R}_{k}=\|r_{k}\|_{L^2(\Omega)} \quad \text{for each} \quad k \in \mathbb{N}, 
\end{equation}
is sufficient for convergence. It means that we choose the stopping index $\mathcal{I}_s$ such that the inequality \eqref{con_cond} is fulfilled, see Hanke and Hansen \cite{hanke1993regularization}. Otherwise, if the given input data $h^\theta$ is exact (without noise, i.e. $\theta=0$), the stopping index can be taken as $\mathcal{I}_s=100.$ The convergence of the proposed method  is addressed in \cite{gilbert1992global, zheng2014recovering} and it has been successfully employed for solving some inverse problems, for instance, see \cite{hrizi2022determination, sun2017identification, yan2019inverse, wei2016inverse}. 
\section{Numerical experiments}\label{num}
In this section, we deal with the numerical implementation of the proposed reconstruction approach and we present some numerical results. To be more precise, we will apply the iterative algorithm {\bf(CGM)} established in the previous section to the numerical treatment of problem $(\mathcal{P}_{op})$, that is, the identification of the initial value $g$ of the initial-boundary value problem \eqref{space_time}. \\

\noindent To begin with, we list the parameters and their values in the numerical calculations as follows:
\begin{itemize}
    \item The domain $\Omega$ is taken to be $\Omega=(-1,1)$.
    \item The final time is fixed to be $T=1$.
    \item The constant $\sigma$  in \eqref{con_cond} is taken to be $\sigma=1.01.$
    \item In the adjoint problem \eqref{addjoint-z}, the Dirac delta
function is approximated by
 $$\delta(t-T)\approx \frac{e^{-(t-T)^2/\eta^2}}{\eta \sqrt{\pi}},$$
where $\eta>0$ is a small positive constant, it is taken to be $\eta=10^{-3}.$ 
 \item We choose the initial guess $g_0$ as a constant, e.g., $g_0 \equiv 1$.
\end{itemize}
\noindent With the exact solution $g_{ex}$, we produce the noisy final data $h^\theta$ by adding a random perturbation, i.e.
$$h^\theta(x):=h(x)+\mu \, h(x)\, \big[2\,\text{\it rand}(\text{size}(h))-1\big],\; \; x\in \overline{\Omega}.$$
Here $\mu \geq 0$ and {\it rand}$(.)$ denote the relative noise level and the uniformly distributed random number in $(0,1)$, respectively. The function $h$ represents the final time value of the solution to problem \eqref{space_time} with initial value $g=g_{ex}$, that is $h(x)=u_{g_{ex}}(x,T)$, $x\in  \overline{\Omega}.$ The corresponding noise level is calculated by $$\theta = \|h^\theta - h\|_{L^2(\Omega)}.$$
In order to evaluate the performance of the conjugate gradient algorithm {\bf (CGM)}, we introduce a $L^2$-error function $E_k$ defined as 
$$E_k=\|g_{ex}-g_k\|_{L^2(\Omega)},$$
which is the error between the exact solution $g_{ex}$ and the reconstructed one $g_k$ at the $k$th iteration.\\

In order to apply the proposed algorithm, we need to solve the forward problem \eqref{space_time} and the adjoint problem \eqref{addjoint-z} numerically. Since it is very difficult to know the eigenpairs $\left\{\left(\varphi_{k}, \lambda_{k}\right)\right\}_{k=1}^{\infty}$ associated to the operator $(-\Delta)^s$ in the explicit forms. Thus, in this study, we use an approximation method, similar to the one in \cite{bensaleh2021inverse} to solve the direct and adjoint problems in each iteration instead of
using \eqref{expansion}, and the series solution of the adjoint problem. Namely, we use a fully discrete approximation based on:
\begin{itemize}
    \item {\it A finite difference scheme:} This method is introduced by Y. Liu et al \cite{lin2007finite} to discretize the Caputo derivative in time. It is proved that this method is of order $2-\alpha$.
    \item {\it The standard $\mathbb{P}_1$ finite element method:} This one is utilized for the space approximation. The authors in \cite{acosta2017short,borthagaray2022fractional} utilized this method to find an approximate solution to the following elliptic problem involving the fractional Laplace operator:
    
    \begin{equation*}
\left\{\begin{aligned}
 (-\Delta)^s u & = f & & \text { in } \Omega, \\
u & = 0 & & \text { in }  \Omega^c.
\end{aligned}\right.
\end{equation*}
They proved that the convergence order of this method depends on the regularity of the right-hand side f, that is
$$
\eth(s):=\left\{
  \begin{array}{rllcl}
    &\frac{1}{2}+\beta&\hbox{if} & f \in H^{\frac{1}{2}-s}(\Omega), \\
    \\
     &\ds 2\beta& \hbox{if}& f \in L^2(\Omega),
  \end{array}
\right.
$$
where $\beta =\min\{s, \frac{1}{2}\}$.
\end{itemize}
In order to illustrate the performance of the proposed numerical approach, we solve the direct problem numerically. Namely, we find the approximate solution to the following problem:
\begin{equation*}
\left\{\begin{aligned}
\partial_t^\alpha u = -(-\Delta)^s u & + F, & & (x, t) \in  \Omega \times(0, T), \\
u & = 0, & & (x, t) \in   \Omega^c \times(0, T), \\
u & = g, & & (x, t) \in  \Omega \times\{0\}.
\end{aligned}\right.
\end{equation*}
By choosing 
$$F(x,t)=1+t^\alpha+\frac{\Gamma(1+\alpha)\, \sqrt{\pi}\,2^{-2s}\, \big(1-x^2\big)^s}{\Gamma(s+1/2)\,\Gamma(s+1)} \text{ and } g(x)=\frac{\sqrt{\pi}\,2^{-2s}\, \big(1-x^2\big)^s}{\Gamma(s+1/2)\,\Gamma(s+1)},$$
then the unique solution to the above problem reads as 
$$u_{ex}(x,t)=(1+t^\alpha)\, \frac{\sqrt{\pi}\,2^{-2s}\, \big(1-x^2\big)^s}{\Gamma(s+1/2)\,\Gamma(s+1)}\, \chi_{(-1,1)}(x), \; \forall (x,t) \in (-1,1)\times (0,T),$$
where $\chi_{(-1,1)}$ denotes the characteristic function of the interval $(-1,1)$.\\

In our computation, we divide the interval $\Omega$ into $N$ equally spaced subintervals, with a mesh size $\Delta x=1/N$. Similarly, the time interval $[0, T]$ will be divided uniformly into $K$ subintervals, and by $\Delta t=1/K$ we denote the time step. In order to illustrate the accuracy of the proposed approximation method, we measure the $L^2-$error $\|u_{ex}-u_{app}\|_{L^2(\Omega)}$ evaluated at the final time $T$. In the following, we examine the time and spatial convergence. A fixed small time step $(\Delta t=1/200)$ is taken to see the spatial convergence
and vice-versa. In Table 1, we show the temporal and spatial convergence rates, indicated in the
column rate (the number in the bracket is the theoretical rate), for different values of $\alpha$ and $s$, which fully confirm the theoretical
results obtained in \cite{lin2007finite} for the time discretization and in \cite{acosta2017short,borthagaray2022fractional} for the spatial one.\\

\begin{table}[h]
    \centering
   \begin{tabular}[b]{|l|c|c|c|c|c|l|}
\hline
\multicolumn{2}{|c|}{\backslashbox{Derivative orders}{Step size}}  & $1/50$ & $1/100$ & $1/200$&$1/400$&{\bf rate} \\
\hline 

\multirow{2}{2cm}{$s=0.5$} & $\alpha=0.3$ & $3.66e-04$&  $1.11e-04$& $3.37e-05$& $1.07e-05$&$1.69$ {\bf (1.7)}\\
\cline{2-7}
 & $\alpha=0.8$ & $2.74e-04$ & $1.20e-04$&$4.93e-05 $&$2.25e-05$&$1.21$ {\bf (1.2)} \\
\hline
\multirow{2}{2cm}{$\alpha=0.5$} & $s=0.2$ & $9.35e-03$&$6.03e-03$&$3.84e-03$&$2.43e-03$ &$0.64$ {\bf (0.7)} \\
\cline{2-7}
 & $s=0.9$ &$2.92e-04$&$1.59e-04$&$8.36e-05$&$4.30e-05$&$0.92$ {\bf (1)} \\
\hline 
\end{tabular}
    \caption{Values of the $L^2-$error function  relative to the step size variation.}
    \label{errors}
\end{table}

In Figure \ref{slope}, we plot the errors in the $L^2-$norm as a function of the (time or space) step sizes. A logarithmic scale has been used for both step-axis and error-axis in these figures.
\begin{figure}[H]
     \centering
     \begin{subfigure}[b]{0.53\textwidth}
         \centering
         \hspace*{.3cm}\includegraphics[width=60mm,height=50mm]{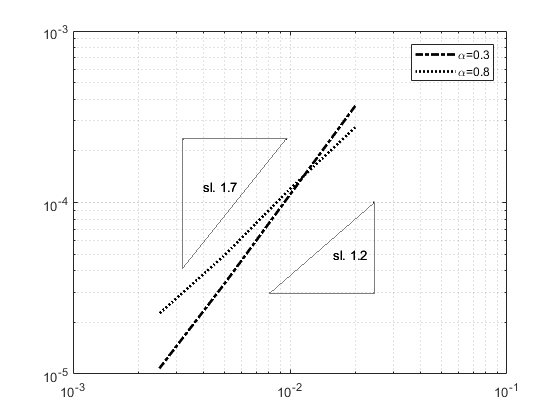}
         \caption{with respect to $\Delta t$}
         
     \end{subfigure}
     \hfill
     \begin{subfigure}[b]{0.46\textwidth}
         \centering
         \hspace*{.3cm}\includegraphics[width=60mm,height=50mm]{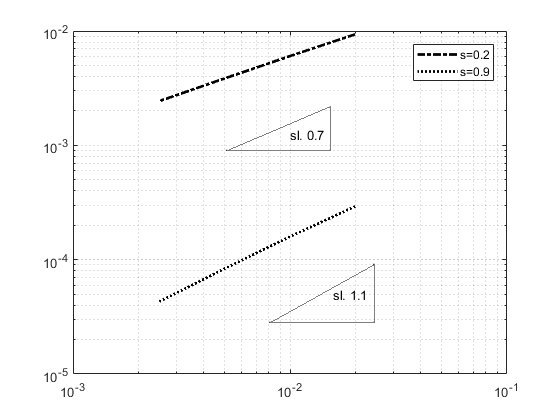}
         \caption{with respect to $\Delta x$}
         
     \end{subfigure}
        \caption{Evolution of the $L^2-$ error functions with respect to the step sizes (sl: slope).}
        \label{slope}
\end{figure}
From Table \ref{errors} and Figure \ref{slope}, one can observe that the numerical results show O$(\Delta t^{2-\alpha})$ and O$(\Delta x^{\eth(s)})$ convergence rates in the $L2-$norm for the temporal and spatial discretization, respectively. Therefore, one can conclude that the numerical and theoretical convergence rates (obtained in \cite{lin2007finite,acosta2017short,borthagaray2022fractional}) are
nearly identical. To this end, we plot in Figure \ref{approx-sol} the variations of the exact  $u_{ex}$ and the approximate solutions $u_{app}$ at time $t=0.5$ for different values of the derivatives orders $\alpha$ and $s$ for showing the quality of the obtained results.

\begin{figure}[H]
     \centering
     \begin{subfigure}[b]{0.53\textwidth}
         \centering
         \hspace*{.3cm}\includegraphics[width=60mm,height=50mm]{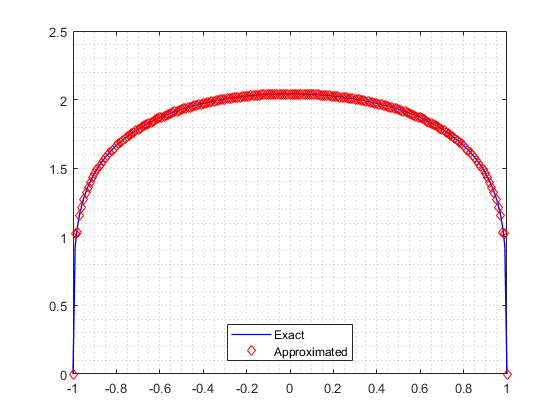}
         \caption{For $\alpha=0.3$ and $s=0.2$}
     \end{subfigure}
     \hfill
     \begin{subfigure}[b]{0.46\textwidth}
         \centering
         \hspace*{.3cm}\includegraphics[width=60mm,height=50mm]{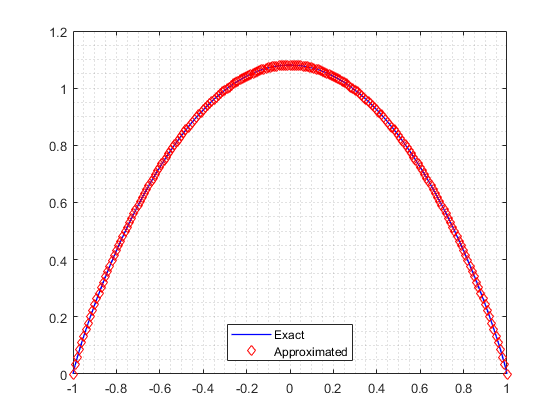}
         \caption{For $\alpha=0.3$ and $s=0.9$}
     \end{subfigure}
     \\
      \begin{subfigure}[b]{0.53\textwidth}
         \centering
         \hspace*{.3cm}\includegraphics[width=60mm,height=50mm]{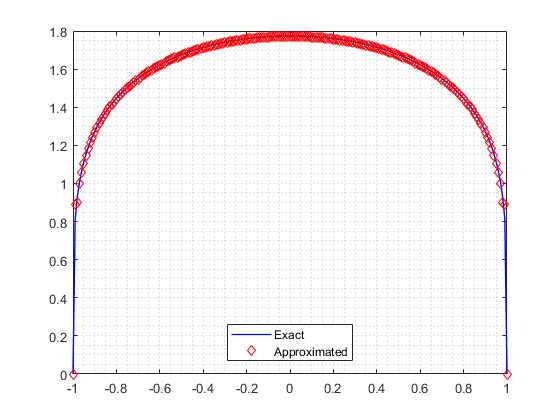}
         \caption{For $\alpha=0.8$ and $s=0.2$}
     \end{subfigure}
     \hfill
     \begin{subfigure}[b]{0.46\textwidth}
         \centering
         \hspace*{.3cm}\includegraphics[width=60mm,height=50mm]{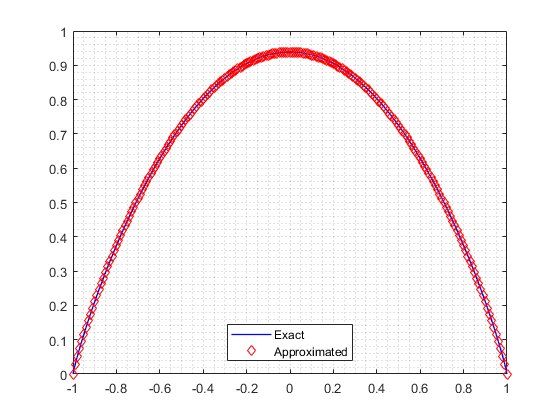}
         \caption{For $\alpha=0.8$ and $s=0.9$}
     \end{subfigure}
        \caption{Exact and approximate solutions for different values of the fractional derivatives orders $\alpha$ and $s$.}
        \label{approx-sol}
\end{figure}
In the rest of the paper, 
the grid size for the time and space variables
 will be taken to be $\Delta t=\Delta x=\frac{1}{100}$.\\

Then we test the performance of the proposed procedure in finding an estimate of the solution of the minimization problem $(\mathcal{P}_{op})$. More precisely, we will discuss the influence of some parameters such as the fractional derivatives orders $\alpha$ and $s$, the regularization parameter $\gamma$, the smoothness of the function to be reconstructed, and  the level of noise $\theta$.
\paragraph*{Example \ref{num}.1}\label{der_order} In this example, we study the influence of the fractional derivatives orders $\alpha$ and $s$ on the quality of the reconstructed results. In doing so, we apply the iterative procedure {\bf (CGM)} to recover
$$g_{ex}(x)=\cos(\pi x)\, \sin(\pi x),$$
with $(\alpha , s) \in  \{0.3, 0.8\} \times \{0.2, 0.9\}$ from an exact final data (i.e., from $h^\theta$, with
$\mu=0$). The results of this test example are illustrated in Figure \ref{ordre}.
\begin{figure}[H]
     \centering
     \begin{subfigure}[b]{0.53\textwidth}
         \centering
         \hspace*{.3cm}\includegraphics[width=70mm,height=60mm]{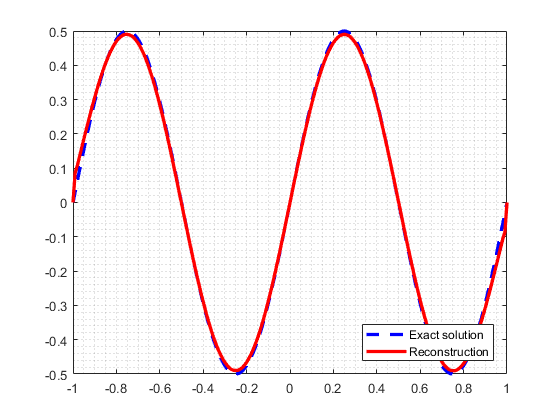}
         \caption{For $\alpha=0.3$ and $s=0.2$}
     \end{subfigure}
     \hfill
     \begin{subfigure}[b]{0.46\textwidth}
         \centering
         \hspace*{.3cm}\includegraphics[width=70mm,height=60mm]{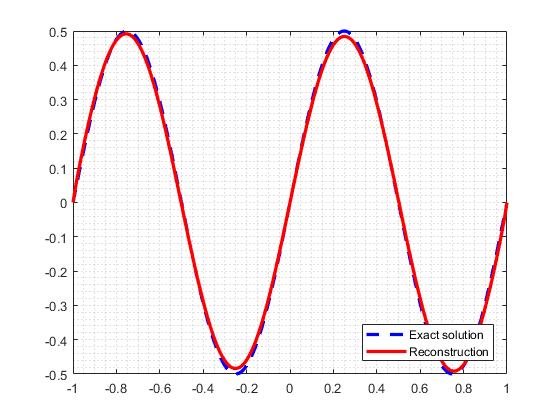}
         \caption{For $\alpha=0.3$ and $s=0.9$}
     \end{subfigure}
     \\
      \begin{subfigure}[b]{0.53\textwidth}
         \centering
         \hspace*{.3cm}\includegraphics[width=70mm,height=60mm]{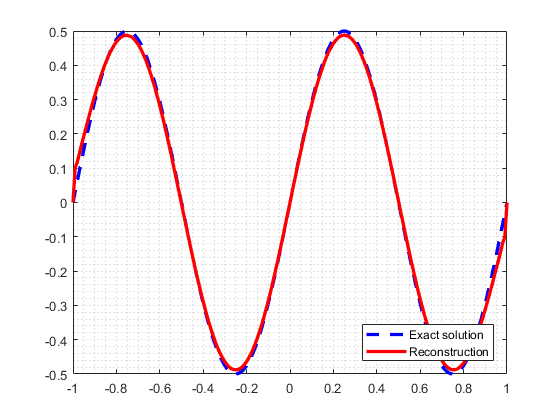}
         \caption{For $\alpha=0.8$ and $s=0.2$}
     \end{subfigure}
     \hfill
     \begin{subfigure}[b]{0.46\textwidth}
         \centering
         \hspace*{.3cm}\includegraphics[width=70mm,height=60mm]{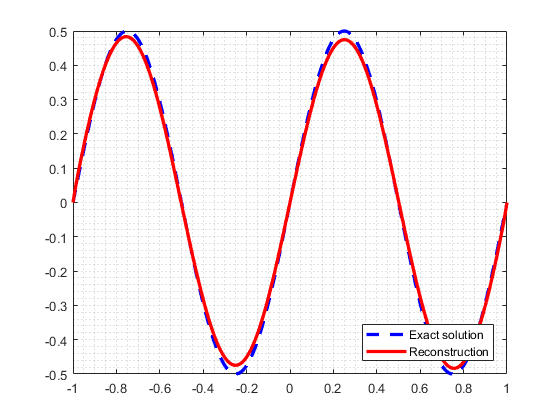}
         \caption{For $\alpha=0.8$ and $s=0.9$}
     \end{subfigure}
        \caption{Exact (blue dashed line) and Reconstructed functions (red line).}
        \label{ordre}
\end{figure}

The choices of $\alpha$ and $s$ in this test and the corresponding
numerical performances are listed in Table \ref{order_per}.
\begin{table}[H]
  \begin{center}
\begin{tabular}{||c| c|c|c| c||} 
 \hline
 $(\alpha,s)$ & $(0.3,0.2)$ & $(0.3,0.9)$ & $(0.8,0.2)$ & $(0.8,0.9)$ \\ 
 \hline
$E_k$ & $0.00654$ & $0.00811$ & $0.00783$ & $0.00827$ \\  
 \hline
\end{tabular}
\label{aa}
\end{center}
\caption{Choices of the fractional derivatives orders $\alpha$ and $s$ along with the errors $E_k$.}
\label{order_per}
\end{table}
From Figure \ref{ordre} and Table \ref{order_per}, one can conclude that the numerical results are quite accurate to the exact solution for each choice of the pairs $(\alpha, s)$. It means that the parameters $\alpha$ and $s$ have no significant influence on the reconstruction results. Therefore, based on the above remarks, the fractional derivatives orders $\alpha$ and $s$ will be fixed to be $(\alpha, s)= (0.5, 0.5)$ in the examples below. 
\paragraph*{Example \ref{num}.2}\label{regu_choice} As it is shown in Section \ref{inv_pbm}, the considered inverse problem is ill-posed in the Hadamard sense and a regularization principle is a key tool to find a stable solution. But  choosing an appropriate (not the optimal) value for the regularization parameter $\gamma$ is a crucial issue. In this study, we use the Tikhonov regularization which is a typical example of an a-priori parameter choice rule since the choice of $\gamma>0$ is made a-priori (depends only on the noise level $\theta>0$). Referring to \cite{galatsanos1991cross, Morozov1}, the regularization parameter $\gamma$ can be chosen such that the consistency condition
$$\lim_{\gamma \to 0} \frac{\theta}{\sqrt{\gamma}}=0$$
is satisfied. Thus, the parameter $\gamma$ can be taken to be $\gamma=10^{-2}\times \theta^{4/5}$. In order to present the importance of the regularization technique for stabilizing our inverse problem, we apply our iterative procedure to reconstruct the initial value $g_{ex}$ for $\gamma=0$ and $\gamma=10^{-2}\times \theta^{4/5}$ by taking a stopping index $\mathcal{I}_s=100$. For this purpose, we test our algorithm in the two following cases:
\begin{enumerate}
    \item[]{\it \bf $(i)$ Slightly noisy data:} In this case, we assume that the measured final data $h^\theta$ is perturbed by low relative noise levels. In Figure \ref{regularization2}, we compare the recovered solutions with the exact one in the cases of $\gamma=0$ and $\gamma=10^{-2}\times \theta^{4/5}$ with $\mu \in \{0.001, 0.005, 0.01\}.$
\begin{figure}[H]
     \centering
     \begin{subfigure}[b]{0.53\textwidth}
         \centering
         \hspace*{.3cm}\includegraphics[width=70mm,height=60mm]{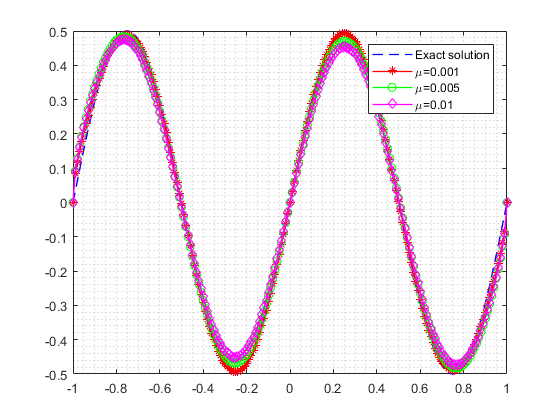}
         \caption{With $\gamma=0$}
         \label{reg_l_0}
     \end{subfigure}
     \hfill
     \begin{subfigure}[b]{0.46\textwidth}
         \centering
         \hspace*{.3cm}\includegraphics[width=70mm,height=60mm]{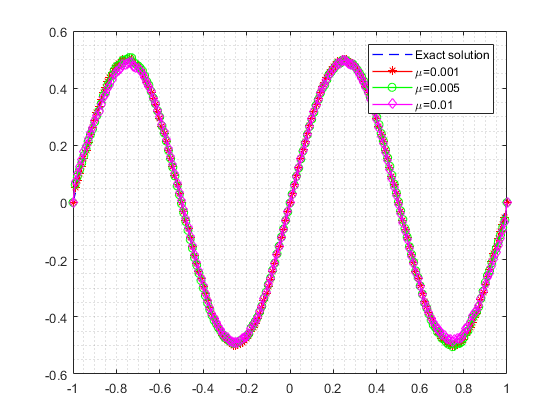}
         \caption{With $\gamma=10^{-2}\times \theta^{4/5}$}
         \label{reg_l_theta}
     \end{subfigure}
        \caption{Exact solution $g_{ex}$ and their reconstructions $g_k$ obtained by {\bf (CGM)} with $\mu \in \{0.001, 0.005, 0.01\}$. Left: with $\gamma=0$; Right: with $\gamma=10^{-2}\times \theta^{4/5}$.}
        \label{regularization2}
\end{figure}

As one can observe from  Figure \ref{regularization2} that the numerical approximations of $g_{ex}$
without using the regularization (see Figure \ref{reg_l_0}) are almost similar to those obtained
when $\gamma=10^{-2}\times \theta^{4/5}$ (see Figure \ref{reg_l_theta}).

\vspace{0.5cm}
    \item[]{\it \bf $(ii)$ Highly noisy data:} In this case, we consider the case where the measured data is perturbed by a high-levels of noise. We illustrate the  comparisons of recovered solutions with the exact one, in Figure \ref{regularization}, but now with $\mu \geq 0.05$, more precisely with relative levels noise $\mu \in \{0.05, 0.1, 0.15\}$. 
\begin{figure}[H]
     \centering
     \begin{subfigure}[b]{0.53\textwidth}
         \centering
         \hspace*{.3cm}\includegraphics[width=70mm,height=60mm]{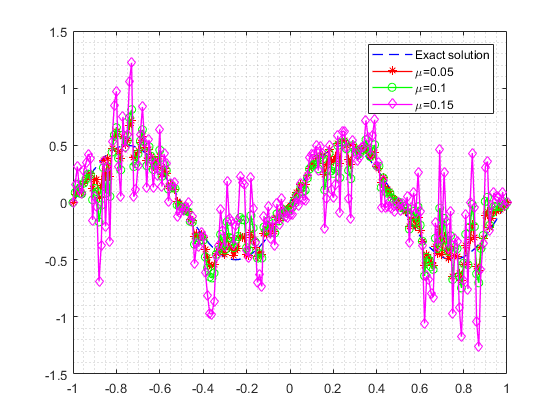}
         \caption{With $\gamma=0$}
         \label{reg_h_0}
     \end{subfigure}
     \hfill
     \begin{subfigure}[b]{0.46\textwidth}
         \centering
         \hspace*{.3cm}\includegraphics[width=70mm,height=60mm]{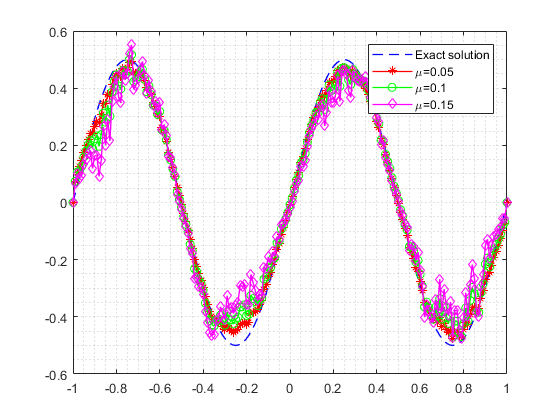}
         \caption{With $\gamma=10^{-2}\times \theta^{4/5}$}
         \label{reg_h_theta}
     \end{subfigure}
        \caption{Exact solution $g_{ex}$ and their reconstructions $g_k$ obtained by {\bf (CGM)} with $\mu \in \{0.05, 0.1, 0.15\}$. Left: with $\gamma=0$; Right: with $\gamma=10^{-2}\times \theta^{4/5}$.}
        \label{regularization}
\end{figure}

From Figure \ref{reg_h_0}, one can conclude that the numerical approximations of $g_{ex}$, without using the regularization (i.e. $\gamma=0$), have somewhat amplitude oscillations and there are very unstable and far away from the exact one. However, by using the regularization technique (i.e. $\gamma=10^{-2}\times \theta^{4/5}$), it can be clearly seen in  Figure \ref{reg_h_theta} that the amplitude of the oscillations decreases and numerical results become more accurate.
\end{enumerate}
We also conclude the following from example 5.2: 
\begin{itemize}
    \item If the  measured data is  slightly perturbed, the proposed algorithm {\bf (CGM)} provides stable numerical solutions to the minimization problem $(\mathcal{P}_{op})$ without the need to use the Tikhonov regularization
term (i.e. $\gamma=0$).
    \item If the final measured data $h^\theta$ is perturbed by high levels of noise, the proposed regularization method is necessary to solve the ill-posed
inverse problem.
\end{itemize}
This result is consistent with the property of ill-posed problems. Thus, in the examples below, the regularization parameter $\gamma$ will be taken to be $\gamma=10^{-2}\times \theta^{4/5}$. \\


 In the following examples, we test the performance of the proposed algorithm to reconstruct two examples of initial values (smooth and nonsmooth functions) by combining the conjugate gradient method with Morozov’s discrepancy principle. For each example, we investigate the convergence of the proposed approach and present the reconstruction results for various choices of the parameter $\mu$.  

\vspace{0.3cm}
\paragraph*{Example \ref{num}.3}  In this example,we test the numerical performance of algorithm {\bf (CGM)} in recovering a smooth initial value given  by 
$$g_{ex}^{sm}(x)=\sin(\pi x)\, e^{-x^2}-cos(\pi x)\, e^{x^2},$$
with various choices of the relative noise level $\mu$ (see Table \ref{ind1}). First, we investigate the convergence of the proposed algorithm and indicate the stooping index $\mathcal{I}_s$ which is determined as 

$$
\mathcal{I}_s=\left\{
  \begin{array}{rllcl}
    &100&\hbox{if} & \mu=0, \\
    \\
     &\ds\inf \mathcal{E}_\theta& \hbox{if}& \mu>0,
  \end{array}
\right.
$$
where the set $\mathcal{E}_\theta$ is given by $\mathcal{E}_\theta:=\big\{k \; \text{such that}\; \mathcal{R}_{k} \leqslant \sigma \theta < \mathcal{R}_{k-1} \big\}$. Figure \ref{err_ex1} shows the convergence of the identification process corresponding to each considered noise level for iteration steps $k=1:100$. To determine the stopping index $\mathcal{I}_s$ for each considered noise level, we illustrate the variations of the residuals $\mathcal{R}_k$  in Figure \ref{res_ex1}. 
\begin{figure}[H]
     \centering
     \begin{subfigure}[b]{0.53\textwidth}
         \centering
         \hspace*{.3cm}\includegraphics[width=70mm,height=60mm]{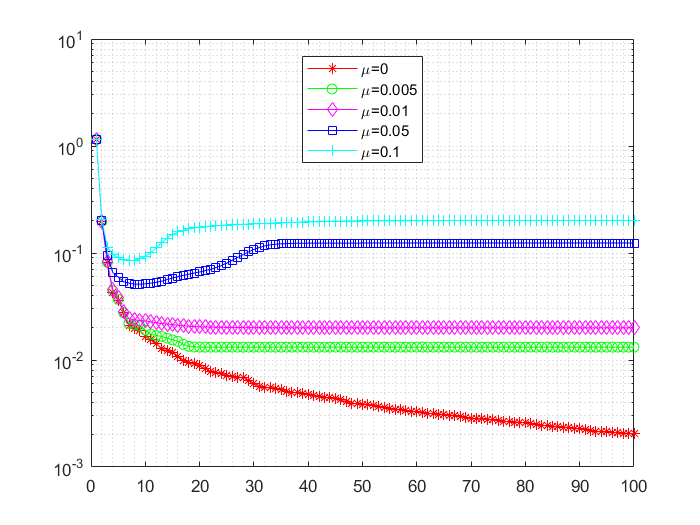}
         \caption{The errors $E_k$}
         \label{err_ex1}
     \end{subfigure}
     \hfill
     \begin{subfigure}[b]{0.46\textwidth}
         \centering
        \hspace*{.3cm} \includegraphics[width=70mm,height=60mm]{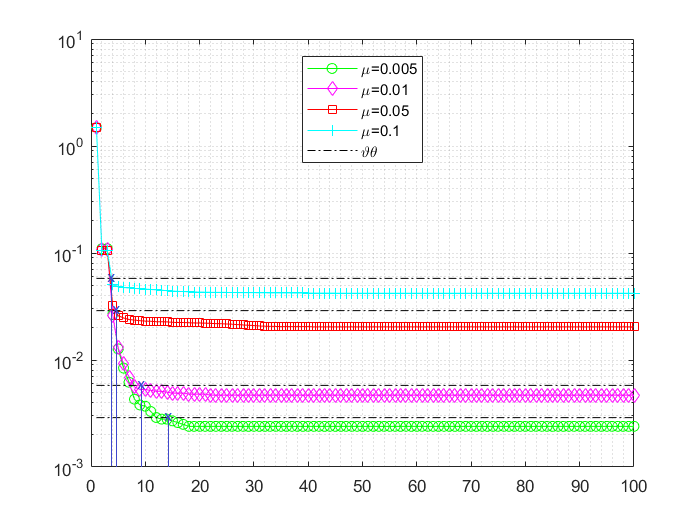}
         \caption{The residuals $\mathcal{R}_k$}
         \label{res_ex1}
     \end{subfigure}
        \caption{The errors $E_k$ and the residuals $\mathcal{R}_k$ associated to the reconstruction of $g_{ex}^{sm}$ for various relative noise levels (x: stopping index).}
        \label{conv_stab_ex1}
\end{figure}
\begin{table}[H]
  \begin{center}
\begin{tabular}{||c| c|c|c| c||} 
 \hline
 $\mu$ & $0.005$ & $0.01$ & $0.05$ & $0.1$ \\ 
 \hline
$\mathcal{I}_s$ & $14$ & $9$ & $5$ & $4$ \\  
 \hline
\end{tabular}
\end{center}
\caption{The obtained stopping indices $\mathcal{I}_s$ associated to the reconstruction of $g_{ex}^{sm}$.}
\label{ind1}
\end{table}
According to the stopping indices listed in Table \ref{ind1},  we illustrate the numerical results by using the discrepancy principle in Figure \ref{result_ex1}. 
    \begin{figure}[H]
\centering
{\includegraphics
[width=70mm,height=60mm]{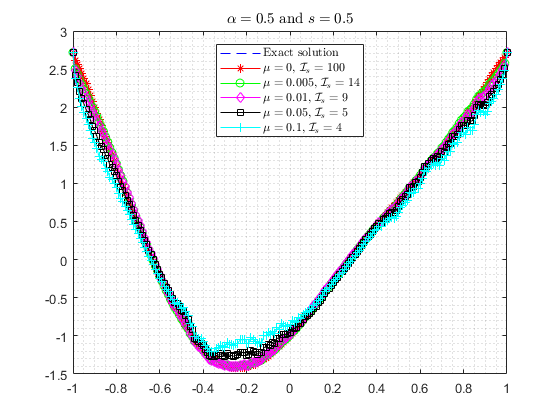} } 
\caption{The numerical results of $g_{ex}^{sm}$ for various relative noise levels.}
\label{result_ex1}
\end{figure}
\vspace{0.3cm}
\paragraph*{Example \ref{num}.4} In this example,  we reconstruct a nonsmooth initial value given by
\begin{equation*} 
 g_{ex}^{nsm}(x)  = \left\{
 \begin{array}{ll}
 0.5,  & \mbox{ } x \in [-1,-0.75) \cup (0.75, 1], \\
 \\
 1, & \mbox{ } x \in [-0.75,-0.25) \cup (0.25, 0.75], \\
 \\
 2, & \mbox{ } x \in [-0.25, 0.25],
\end{array}
\right. 
\end{equation*}
with the same relative noise levels chosen in Example \ref{num}.3. We present the convergence of the estimated solutions to the exact one in Figure \ref{conv_stab_ex2}. More precisely, we illustrate
the approximation errors $E_k$ and residuals $\mathcal{R}_k$ for iteration steps $k=1:100$ with
the considered choices of the parameters $\mu$ in Figures \ref{err_ex2} and \ref{res_ex2}, respectively.

\begin{figure}[H]
     \centering
     \begin{subfigure}[b]{0.53\textwidth}
         \centering
         \hspace*{.3cm}\includegraphics[width=70mm,height=60mm]{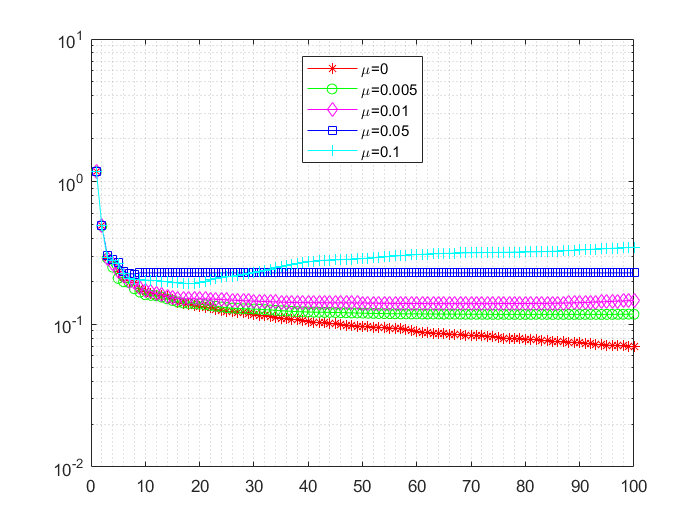}
         \caption{The errors $E_k$}
         \label{err_ex2}
     \end{subfigure}
     \hfill
     \begin{subfigure}[b]{0.46\textwidth}
         \centering
         \hspace*{.3cm}\includegraphics[width=70mm,height=60mm]{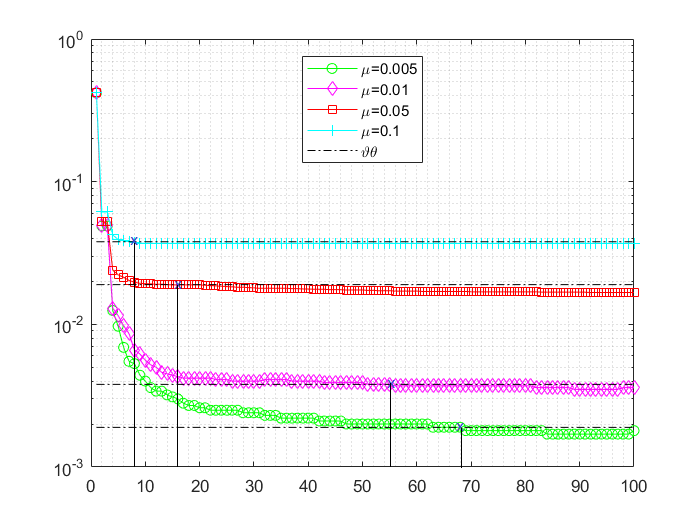}
         \caption{The residuals $\mathcal{R}_k$}
         \label{res_ex2}
     \end{subfigure}
        \caption{The errors $E_k$ and the residuals $\mathcal{R}_k$ associated to the reconstruction of $g_{ex}^{nsm}$ for various relative noise levels (x: stopping index).}
        \label{conv_stab_ex2}
\end{figure}
The obtained stopping indices for each choice $\mu$ which are determined with the help of curves presented in Figure \ref{res_ex2} are listed in Table \ref{ind2}.
\begin{table}[H]
  \begin{center}
\begin{tabular}{||c| c|c|c| c||} 
 \hline
 $\mu$ & $0.005$ & $0.01$ & $0.05$ & $0.1$ \\ 
 \hline
$\mathcal{I}_s$ & $68$ & $55$ & $16$ & $8$ \\  
 \hline
\end{tabular}
\end{center}
\caption{The obtained stopping indices $\mathcal{I}_s$ associated to the reconstruction of $g_{ex}^{nsm}$.}
\label{ind2}
\end{table}

In Figure \ref{result_ex2}, we illustrate the comparisons of recovered solutions with the exact ones for each considered relative noise level $\mu$.
    \begin{figure}[H]
\centering
{\includegraphics
[width=70mm,height=60mm]{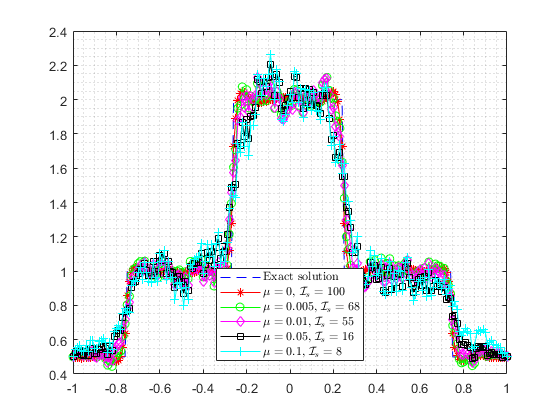} } 
\caption{The numerical results of $g_{ex}^{nsm}$ for various relative noise levels.}
\label{result_ex2}
\end{figure}
We  have the following conclusions from the examples 5.3 and 5.4: 
\begin{itemize}
\item From Figure \ref{result_ex2}, one can observe that the presence of the singular points affects the accuracy of the reconstructed initial value $g_{ex}^{nsm}$. Indeed, the efficiency of our iterative procedure decreases sharply near these points. In comparison with the obtained results in Example 5.3 (see Figure \ref{result_ex1}), one can deduce that the used proposed method works better for the
reconstruction of a smooth function than the identifiability of a nonsmooth
one.
  \item As we can see from Figures \ref{err_ex1} and \ref{err_ex2}, with the increasing of noise levels in the input data $h^\theta$, the approximation error $E_k$ is increasing. Moreover, it can be seen that, for a relative noise $\mu>0$,  the computed errors become constants or have slightly increased after a few iterations. Thus we have to stop at a suitable step. From Figures \ref{res_ex1} and \ref{res_ex2}, we see that Morozov’s discrepancy principle provides the stopping steps when the condition \eqref{con_cond} is fulfilled  for each relative noise level $\mu>0$.
    \item From Figures \ref{result_ex1} and \ref{result_ex2}, we can see that the proposed numerical procedure {\bf (CGM)} is robust with respect to noise and able to provide reasonably good reconstruction results. Indeed, we can see
that the numerical results for the two considered examples are quite accurate up to $1\%$ noise added in the final data $u(x, T)$. However, when the relative noise level $\mu$ exceeds $0.05$, we can observe some fictitious oscillations that affect the accuracy and the quality of the obtained reconstruction results. Namely, in this case, the solution found might move further away from the true solution. Thus, when using highly noisy data,
the ill-posedness of the reconstruction problem becomes more severe. As we mentioned above that in order to ameliorate these results (find more stable solutions), we need to determine the optimal regularization parameter which depends directly on the noise level. In fact, choosing an appropriate regularization parameter
is a difficult problem. The theoretical aspect of this issue has not
been addressed so far. It is still an open question. It will be the subject of a forthcoming work.\quad \\

\end{itemize}
\section{Concluding Remarks}\label{con}

In this paper, we studied an inverse problem for a space-time fractional diffusion equation. Namely, the considered inverse problem is to determine the initial value from the noise measurement of the final time. We proved a uniqueness result for the inverse problem with the help of the eigenfunction expansion method. The inverse problem is reformulated as a regularized optimization problem. By using the adjoint method, a first-order optimality condition has been derived. Based on this condition, the unknown term has been characterized as the solution of a regularized variational problem. An efficient and accurate iterative reconstruction procedure has been developed and implemented to solve the regularized variational problem; that is, the conjugate gradient method combined with Morozov’s discrepancy principle. The influences of some parameters (the fractional derivative orders, the regularization parameter, and the noisy data) on the reconstruction results have been discussed. The efficiency and accuracy of the proposed numerical procedure have been justified by some numerical simulations.\\

The authors of this paper plan to study a system of partial differential equations involving time-fractional derivatives and nonlinear diffusion operators such as $A_1(\nabla u):=-\nabla. \big(k(|\nabla u|^{2}) \nabla u \big)$ or $A_2(u):= -\nabla. \big(k(u) \nabla u \big)$. First, we study the well-posedness of the direct problem, then we define the inverse problem.  The inverse problem may be either an inverse parameter problem, or an inverse coefficient problem, or an inverse source problem. Then we study the existence and uniqueness of the solution and solve the inverse problem by the method presented in this paper.

\section{Acknowledgments} \label{ack}

\noindent The authors would like to thank Dr. J. P. Borthagay for the assistance in providing the computational experiments.

\bibliographystyle{abbrv}
\bibliography{biblio}

\begin{thebibliography}{10}

\bibitem{abdelwahed2022inverse}
M.~Abdelwahed, M.~BenSaleh, N.~Chorfi, and M.~Hassine.
\newblock An inverse problem study related to a fractional diffusion equation.
\newblock {\em Journal of Mathematical Analysis and Applications},
  512(2):126145, 2022.

\bibitem{abdulkerimov1977regularizatio}
L.~S. Abdulkerimov.
\newblock Regularization of an ill-posed cauchy problem for evolution equations
  in a banach space.
\newblock {\em Azerbaidzan. Gos. Univ. Ucen. Zap. Fiz. Mat}, 1:32--36, 1977.

\bibitem{acosta2017short}
G.~Acosta, F.~M. Bersetche, and J.~P. Borthagaray.
\newblock A short fe implementation for a 2d homogeneous dirichlet problem of a
  fractional laplacian.
\newblock {\em Computers \& Mathematics with Applications}, 74(4):784--816,
  2017.

\bibitem{acosta2019finite}
G.~Acosta, F.~M. Bersetche, and J.~P. Borthagaray.
\newblock Finite element approximations for fractional evolution problems.
\newblock {\em Fractional Calculus and Applied Analysis}, 22(3):767--794, 2019.

\bibitem{adams1975sobolev}
R.~A. Adams.
\newblock Sobolev spaces new york.
\newblock {\em San Francisco. Lodon: Academic Press}, 1:975, 1975.

\bibitem{agrawal2007fractional}
O.~Agrawal.
\newblock Fractional variational calculus in terms of riesz fractional
  derivatives.
\newblock {\em Journal of Physics A: Mathematical and Theoretical},
  40(24):6287, 2007.

\bibitem{almeida2011necessary}
R.~Almeida and D.~F. Torres.
\newblock Necessary and sufficient conditions for the fractional calculus of
  variations with caputo derivatives.
\newblock {\em Communications in Nonlinear Science and Numerical Simulation},
  16(3):1490--1500, 2011.

\bibitem{bensalah2021inverse}
M.~BenSalah and M.~Hassine.
\newblock Inverse source problem for a diffusion equation involving the
  fractional spectral laplacian.
\newblock {\em Mathematical Methods in the Applied Sciences}, 44(1):917--936,
  2021.

\bibitem{bensaleh2021inverse}
M.~BenSaleh and H.~Maatoug.
\newblock Inverse source problem for a space-time fractional diffusion
  equation.
\newblock {\em Ricerche di Matematica}, pages 1--33, 2021.

\bibitem{borthagaray2022fractional}
J.~P. Borthagaray, W.~Li, and R.~H. Nochetto.
\newblock Fractional elliptic problems on lipschitz domains: Regularity and
  approximation.
\newblock {\em arXiv preprint arXiv:2212.14070}, 2022.

\bibitem{nane}
Z.-Q. Chen, M.~M. Meerschaert, and E.~Nane.
\newblock Space--time fractional diffusion on bounded domains.
\newblock {\em Journal of Mathematical Analysis and Applications},
  393(2):479--488, 2012.

\bibitem{y1}
J.~Cheng, J.~Nakagawa, M.~Yamamoto, and T.~Yamazaki.
\newblock Uniqueness in an inverse problem for a one-dimensional fractional
  diffusion equation.
\newblock {\em Inverse problems}, 25(11):115002, 2009.

\bibitem{di2012hitchhiker}
E.~Di~Nezza, G.~Palatucci, and E.~Valdinoci.
\newblock Hitchhiker's guide to the fractional sobolev spaces.
\newblock {\em Bulletin des sciences math{\'e}matiques}, 136(5):521--573, 2012.

\bibitem{diethelm2010analysis}
K.~Diethelm.
\newblock The analysis of fractional differential equations, volume 2004 of
  lecture notes in mathematics, 2010.

\bibitem{dipierro2017nonlocal}
S.~Dipierro, X.~Ros-Oton, and E.~Valdinoci.
\newblock Nonlocal problems with neumann boundary conditions.
\newblock {\em Revista Matem{\'a}tica Iberoamericana}, 33(2):377--416, 2017.

\bibitem{du2013nonlocal}
Q.~Du, M.~Gunzburger, R.~B. Lehoucq, and K.~Zhou.
\newblock A nonlocal vector calculus, nonlocal volume-constrained problems, and
  nonlocal balance laws.
\newblock {\em Mathematical Models and Methods in Applied Sciences},
  23(03):493--540, 2013.

\bibitem{engl1996regularization}
H.~W. Engl, M.~Hanke, and A.~Neubauer.
\newblock {\em Regularization of inverse problems}, volume 375.
\newblock Springer Science \& Business Media, 1996.

\bibitem{feng}
X.~Feng, P.~Li, and X.~Wang.
\newblock An inverse random source problem for the time-fractional diffusion
  equation driven by a fractional brownian motion.
\newblock {\em Inverse Problems}, 36(4):045008, 2020.

\bibitem{galatsanos1991cross}
N.~P. Galatsanos and A.~K. Katsaggelos.
\newblock Cross-validation and other criteria for estimating the regularizing
  parameter.
\newblock In {\em [Proceedings] ICASSP 91: 1991 International Conference on
  Acoustics, Speech, and Signal Processing}, pages 3021--3024. IEEE, 1991.

\bibitem{gilbert1992global}
J.~C. Gilbert and J.~Nocedal.
\newblock Global convergence properties of conjugate gradient methods for
  optimization.
\newblock {\em SIAM Journal on optimization}, 2(1):21--42, 1992.

\bibitem{gorenflo1997fractional}
R.~Gorenflo and F.~Mainardi.
\newblock {\em Fractional calculus: integral and differential equations of
  fractional order}.
\newblock Springer, 1997.

\bibitem{grubb2015spectral}
G.~Grubb.
\newblock Spectral results for mixed problems and fractional elliptic
  operators.
\newblock {\em Journal of Mathematical Analysis and Applications},
  421(2):1616--1634, 2015.

\bibitem{ulusoy1}
N.~Guerngar, E.~Nane, R.~Tinaztepe, S.~Ulusoy, and H.~W. Van~Wyk.
\newblock Simultaneous inversion for the fractional exponents in the space-time
  fractional diffusion equation.
\newblock {\em Fractional Calculus and Applied Analysis}, 24(3):818--847, 2021.

\bibitem{hadamard1902problemes}
J.~Hadamard.
\newblock Sur les probl{\`e}mes aux d{\'e}riv{\'e}es partielles et leur
  signification physique.
\newblock {\em Princeton university bulletin}, pages 49--52, 1902.

\bibitem{hanke1993regularization}
M.~Hanke and P.~C. Hansen.
\newblock Regularization methods for large-scale problems.
\newblock {\em Surv. Math. Ind}, 3(4):253--315, 1993.

\bibitem{hrizi2022determination}
M.~Hrizi, M.~BenSalah, and M.~Hassine.
\newblock Determination of the initial density in nonlocal diffusion from final
  time measurements.
\newblock {\em Discrete \& Continuous Dynamical Systems-Series S}, 15(6), 2022.

\bibitem{jiang2020numerical}
D.~Jiang, Y.~Liu, and D.~Wang.
\newblock Numerical reconstruction of the spatial component in the source term
  of a time-fractional diffusion equation.
\newblock {\em Advances in Computational Mathematics}, 46:1--24, 2020.

\bibitem{rundell}
B.~Jin and W.~Rundell.
\newblock An inverse problem for a one-dimensional time-fractional diffusion
  problem.
\newblock {\em Inverse Problems}, 28(7):075010, 2012.

\bibitem{karapinar2020identifying}
E.~Karapinar, D.~Kumar, R.~Sakthivel, N.~H. Luc, and N.~Can.
\newblock Identifying the space source term problem for time-space-fractional
  diffusion equation.
\newblock {\em Advances in difference equations}, 2020(1):1--23, 2020.

\bibitem{KU}
K.~H. Karlsen and S.~Ulusoy.
\newblock Stability of entropy solutions for levy mixed hyperbolic-parabolic
  equations.
\newblock {\em arXiv preprint arXiv:0902.0538}, 2009.

\bibitem{kilbas2006theory}
A.~A. Kilbas, H.~M. Srivastava, and J.~J. Trujillo.
\newblock {\em Theory and applications of fractional differential equations},
  volume 204.
\newblock elsevier, 2006.

\bibitem{kozlov1989iterative}
V.~A. Kozlov and V.~G. Maz'ya.
\newblock Iterative procedures for solving ill-posed boundary value problems
  that preserve the differential equations.
\newblock {\em Algebra i Analiz}, 1(5):144--170, 1989.

\bibitem{lattes1969method}
R.~Latt{\`e}s and J.-L. Lions.
\newblock The method of quasi-reversibility: applications to partial
  differential equations.
\newblock Technical report, 1969.

\bibitem{ym5}
G.~Li, D.~Zhang, X.~Jia, and M.~Yamamoto.
\newblock Simultaneous inversion for the space-dependent diffusion coefficient
  and the fractional order in the time-fractional diffusion equation.
\newblock {\em Inverse Problems}, 29(6):065014, 2013.

\bibitem{oleg}
Z.~Li, O.~Y. Imanuvilov, and M.~Yamamoto.
\newblock Uniqueness in inverse boundary value problems for fractional
  diffusion equations.
\newblock {\em Inverse Problems}, 32(1):015004, 2015.

\bibitem{lin2007finite}
Y.~Lin and C.~Xu.
\newblock Finite difference/spectral approximations for the time-fractional
  diffusion equation.
\newblock {\em Journal of computational physics}, 225(2):1533--1552, 2007.

\bibitem{lischke2020fractional}
A.~Lischke, G.~Pang, M.~Gulian, F.~Song, C.~Glusa, X.~Zheng, Z.~Mao, W.~Cai,
  M.~M. Meerschaert, M.~Ainsworth, et~al.
\newblock What is the fractional laplacian? a comparative review with new
  results.
\newblock {\em Journal of Computational Physics}, 404:109009, 2020.

\bibitem{y2}
J.~Liu and M.~Yamamoto.
\newblock A backward problem for the time-fractional diffusion equation.
\newblock {\em Applicable Analysis}, 89(11):1769--1788, 2010.

\bibitem{ng1}
N.~H. Luc, D.~Kumar, L.~D. Long, and H.~T.~K. Van.
\newblock Final value problem for parabolic equation with fractional laplacian
  and kirchhoff's term.
\newblock {\em Journal of Function Spaces}, 2021:1--12, 2021.

\bibitem{luchko}
Y.~Luchko.
\newblock Initial-boundary-value problems for the one-dimensional
  time-fractional diffusion equation.
\newblock {\em Fractional Calculus and Applied Analysis}, 15(1):141--160, 2012.

\bibitem{MBSB}
M.~M. Meerschaert, D.~A. Benson, H.-P. Scheffler, and B.~Baeumer.
\newblock Stochastic solution of space-time fractional diffusion equations.
\newblock {\em Physical Review E}, 65(4):041103, 2002.

\bibitem{morozov2012methods}
V.~A. Morozov.
\newblock {\em Methods for solving incorrectly posed problems}.
\newblock Springer Science \& Business Media, 2012.

\bibitem{Morozov1}
V.~A. Morozov, Z.~Nashed, and A.~B. Aries.
\newblock Methods for solving incorrectly posed problems.
\newblock {\em (New York: Springer)}, 1984.

\bibitem{11}
I.~Podlubny.
\newblock Fractional differential equations, 198 academic press.
\newblock {\em San Diego, California, USA}, 1999.

\bibitem{podlubny1999fractional}
I.~Podlubny.
\newblock Fractional differential equations, volume 198 of mathematics in
  science and engineering, acad, 1999.

\bibitem{ros2014local}
X.~Ros-Oton and J.~Serra.
\newblock Local integration by parts and pohozaev identities for higher order
  fractional laplacians.
\newblock {\em arXiv preprint arXiv:1406.1107}, 2014.

\bibitem{y4}
K.~Sakamoto and M.~Yamamoto.
\newblock Initial value/boundary value problems for fractional diffusion-wave
  equations and applications to some inverse problems.
\newblock {\em Journal of Mathematical Analysis and Applications},
  382(1):426--447, 2011.

\bibitem{y3}
K.~Sakamoto and M.~Yamamoto.
\newblock Inverse source problem with a final overdetermination for a
  fractional diffusion equation.
\newblock {\em Math. Control Relat. Fields}, 1(4):509--518, 2011.

\bibitem{sun2017identification}
L.~Sun and T.~Wei.
\newblock Identification of the zeroth-order coefficient in a time fractional
  diffusion equation.
\newblock {\em Applied Numerical Mathematics}, 111:160--180, 2017.

\bibitem{salih1}
S.~Tatar, R.~T{\i}naztepe, and S.~Ulusoy.
\newblock Simultaneous inversion for the exponents of the fractional time and
  space derivatives in the space-time fractional diffusion equation.
\newblock {\em Applicable Analysis}, 95(1):1--23, 2016.

\bibitem{salih2}
S.~Tatar and S.~Ulusoy.
\newblock An inverse source problem for a one-dimensional space--time
  fractional diffusion equation.
\newblock {\em Applicable Analysis}, 94(11):2233--2244, 2015.

\bibitem{wang2015quasi}
J.-G. Wang and T.~Wei.
\newblock Quasi-reversibility method to identify a space-dependent source for
  the time-fractional diffusion equation.
\newblock {\em Applied Mathematical Modelling}, 39(20):6139--6149, 2015.

\bibitem{wei2016inverse}
T.~Wei, X.~Li, and Y.~Li.
\newblock An inverse time-dependent source problem for a time-fractional
  diffusion equation.
\newblock {\em Inverse Problems}, 32(8):085003, 2016.

\bibitem{yan2019inverse}
X.~B. Yan and T.~Wei.
\newblock Inverse space-dependent source problem for a time-fractional
  diffusion equation by an adjoint problem approach.
\newblock {\em Journal of Inverse and Ill-posed Problems}, 27(1):1--16, 2019.

\bibitem{yz1}
Y.~Zhang and X.~Xu.
\newblock Inverse source problem for a fractional diffusion equation.
\newblock {\em Inverse problems}, 27(3):035010, 2011.

\bibitem{zheng2014recovering}
G.-H. Zheng and T.~Wei.
\newblock Recovering the source and initial value simultaneously in a parabolic
  equation.
\newblock {\em Inverse Problems}, 30(6):065013, 2014.

\end{thebibliography}
\end{document}